\documentclass[10pt,journal,compsoc]{IEEEtran}

\ifCLASSOPTIONcompsoc
\usepackage[nocompress]{cite}
\else
\usepackage{cite}
\fi

\usepackage{graphicx}
\usepackage{subfigure} 
\usepackage{amsmath}
\usepackage{amsthm}
\usepackage{amssymb}
\usepackage{multirow}

\usepackage{color}
\usepackage{algorithm}
\usepackage{algorithmic}

\usepackage{url}

\usepackage{array}
\newcolumntype{L}[1]{>{\raggedright\let\newline\\\arraybackslash\hspace{0pt}}m{#1}}
\newcolumntype{C}[1]{>{\centering\let\newline  \\\arraybackslash\hspace{0pt}}m{#1}}
\newcolumntype{R}[1]{>{\raggedleft\let\newline \\\arraybackslash\hspace{0pt}}m{#1}}

\usepackage[mathscr]{euscript}

\def\O{\Omega}
\def\l{\lambda}
\def\tZ{\breve{Z}}
\def\svt{\text{SVT}}

\newcommand{\R}{\mathbb{R}}

\newcommand{\SO}[1]{P_{\Omega}(#1) }
\newcommand{\SOO}[1]{P_{\Omega^\bot}(#1)}
\newcommand{\FN}[1]{\left\| #1 \right\|_F}
\newcommand{\NM}[2]{\| #1 \|_{#2}}

\newcommand{\Diag}[1]{\text{Diag}(#1)}

\newcommand{\prox}[2]{\text{prox}_{#1}(#2)}

\newcommand{\vect}[1]{{\boldsymbol{\mathbf{#1}}}}
\newcommand{\ten}[1]{{\vect{\mathscr{{#1}}}}}
\newcommand{\tX}{\ten{X}}
\newcommand{\tY}{\ten{Y}}
\newcommand{\ip}[1]{{\left\langle  #1 \right\rangle}}

\newcommand{\Span}[1]{\text{span}(#1)}
\newcommand{\tr}[1]{\,\text{tr}(#1)}

\newtheorem{theorem}{Theorem}[section]
\newtheorem{lemma}[theorem]{Lemma}
\newtheorem{proposition}[theorem]{Proposition}

\newtheorem{definition}{Definition}[section]

\begin{document}
%
\title{Accelerated and Inexact Soft-Impute for Large-Scale Matrix and Tensor Completion}
%
%
%

\author{
Quanming Yao \IEEEmembership{Member,~IEEE} and James T. Kwok \IEEEmembership{Fellow,~IEEE}
\IEEEcompsocitemizethanks{
\IEEEcompsocthanksitem Q. Yao is with 4Paradigm Inc, Beijing, China. 
E-mail: yaoquanming@4Paradigm.com
\IEEEcompsocthanksitem James T. Kwok is with the Department of Computer Science and Engineering,
Hong Kong University of Science and Technology, Clear Water Bay,
Hong Kong.
E-mail: jamesk@cse.ust.hk. }
}

%
%

\markboth{IEEE Transactions on Knowledge and Data Engineering}%
{Shell \MakeLowercase{\textit{et al.}}: Bare Demo of IEEEtran.cls for IEEE Journals}
%



\IEEEtitleabstractindextext{
\begin{abstract}
Matrix and tensor completion 
aim to recover a low-rank matrix / tensor from limited observations and have been
commonly used in 
applications
such as recommender systems and multi-relational data mining.
A state-of-the-art matrix completion algorithm is Soft-Impute, which exploits the
special ``sparse plus low-rank" structure of the matrix iterates to allow efficient SVD in each iteration.
Though Soft-Impute is a proximal algorithm,
it is generally believed that acceleration
destroys the special structure and is thus not useful.
In this paper, 
we show that Soft-Impute can indeed be accelerated without comprising this structure.
To further reduce the iteration time complexity,
we propose an approximate singular value thresholding scheme based on the power method.
Theoretical analysis shows that the proposed algorithm still enjoys the fast $O(1/T^2)$
convergence rate of accelerated proximal algorithms.
We also extend the proposed algorithm to tensor completion
with the scaled latent nuclear norm regularizer.
We show that a similar ``sparse plus low-rank'' structure also exists, leading to low iteration complexity 
and fast $O(1/T^2)$ convergence rate.
Besides, 
the proposed algorithm can be further extended to 
nonconvex low-rank regularizers,
which have better empirical performance than the convex nuclear norm regularizer.
Extensive experiments demonstrate that the proposed algorithm is much faster than Soft-Impute and other state-of-the-art matrix and tensor completion algorithms.
\end{abstract}

\begin{IEEEkeywords}
	Matrix Completion, Tensor Completion, Collaborative Filtering, Link Prediction, Proximal Algorithms
\end{IEEEkeywords}
}

\maketitle

\IEEEpeerreviewmaketitle

\section{Introduction}

\IEEEPARstart{M}{atrices}
are common place in data mining applications.
For example,
in recommender systems,
the ratings data 
can be represented as a sparsely observed
user-item
matrix \cite{koren2008factorization}.
In social networks, user interactions can be modeled by an adjacency matrix \cite{kim2011,chiang2014prediction}.  
Matrices also appear in applications such as image processing \cite{liu2013tensor,gu2017weighted,lu2016nonconvex},
question answering \cite{zhao2015expert} and large scale classification \cite{fan2018accelerated}.

Due to limited feedback from users,
these matrices are usually not fully observed.
For example, 
users may only 
give opinions on 
very few items
in a recommender system.
As the rows/columns are usually related to each other, the
low-rank matrix assumption is particularly useful to 
capture such relatedness, and
low-rank matrix completion has become a powerful tool 
to predict missing values in these matrices.
Sound recovery guarantee \cite{candes2009exact} and good empirical performance
\cite{koren2008factorization} have been obtained.

However, directly minimizing the matrix rank 
is NP-hard \cite{candes2009exact}.
To alleviate this problem,
the nuclear norm (which is the sum of singular values)
is often used instead.
It is known that the nuclear norm is the tightest convex lower bound of the rank \cite{candes2009exact}.
Specifically,
consider an 
$m \times n$ 
matrix
$O$
(without loss of generality, we assume that $m \ge n$), with
positions of the observed entries indicated by
$\Omega \in \{0,1\}^{m \times n}$, where
$\Omega_{ij}=1$ if $O_{ij}$ is observed, and 0 otherwise.
The matrix completion
tries to find a low-rank matrix $X$ by solving following optimization problem:
\begin{equation} \label{eq:mc}
\min_X
\frac{1}{2}\NM{\SO{X - O}}{F}^2 + \lambda \NM{X}{*},
\end{equation} 
where
$[\SO{A}]_{ij} =
A_{ij}$ if  $\Omega_{ij} = 1$, and
0 otherwise;
and $\NM{\cdot}{*}$ is the nuclear norm.
Though the nuclear norm is only a surrogate of the matrix rank,
there are theoretical guarantees that
the underlying matrix
can be exactly recovered 
\cite{candes2009exact}.

Computationally, 
though the nuclear norm is nonsmooth,
problem~\eqref{eq:mc}
can be solved by various optimization tools.
An early attempt is based on reformulating (\ref{eq:mc}) as a semidefinite program (SDP)
\cite{candes2009exact}. However, SDP solvers have large time and space complexities, and 
are only suitable for small data sets. For 
large-scale matrix completion, 
singular value thresholding (SVT) algorithm \cite{cai2010singular} 
pioneered the use of first-order
methods .
However, 
a singular value decomposition (SVD) is required
in each SVT iteration.
This takes $O(m n^2)$ time
and
can be computationally expensive.
In \cite{toh2010accelerated},
this is reduced to 
a partial SVD by computing only the leading singular values/vectors using PROPACK (a variant
of the Lanczos algorithm) \cite{larsen1998lanczos}.
Another major breakthrough is made by the Soft-Impute algorithm \cite{mazumder2010spectral}, which 
utilizes a special ``sparse plus low-rank'' structure associated with the SVT to efficiently compute the SVD.
Empirically, this allows Soft-Impute to perform matrix completion on the entire
\textit{Netflix} data set.
The SVT algorithm can also be viewed as a proximal algorithm
\cite{tibshirani2010}.
Hence, it converges with a $O(1/T)$ rate, where $T$ is the number of iterations
\cite{beck2009fast}.
Later, this is further ``accelerated", and
the convergence rate is improved to $O(1/T^2)$
\cite{ji2009accelerated,toh2010accelerated}.
However,
Tibshirani
\cite{tibshirani2010} suggested that this is not
useful, as the special
``sparse plus low-rank''
structure crucial to the efficiency of Soft-Impute no longer exist.
In other words, the gain in convergence rate is more than compensated by the increase in
iteration time complexity.

In this paper, 
we 
show that accelerating
Soft-Impute is indeed possible while still
preserving the ``sparse plus low-rank'' structure.
To further reduce the iteration time complexity,
instead of computing SVT exactly using PROPACK \cite{toh2010accelerated,mazumder2010spectral},
we propose an approximate SVT scheme based on the power method \cite{halko2011finding}.
Though the SVT obtained 
in each iteration
is only approximate, 
we show that convergence can still be as fast as performing exact SVT.
Hence, the resultant algorithm has low iteration complexity and fast $O(1/T^2)$ convergence rate.
To further boost performance,
we 
extend
the post-processing procedure  in
\cite{mazumder2010spectral} to any smooth convex loss function.
The proposed algorithm is also
extended 
for nonconvex low-rank regularizers, such as the truncated nuclear norm \cite{hu2013fast} and log-sum-penalty \cite{candes2009exact}.
	which can give better 
Empirically, these
nonconvex low-rank regularizers have better
performance than the convex nuclear norm regularizer.

Besides matrices,
tensors have also been commonly used to describe the linear and
multilinear relationships in the data \cite{kolda2009tensor,tomioka2010estimation,acar2011scalable,liu2013tensor}.
Analogous to matrix completion,
tensor completion can also be solved by convex optimization algorithms. 
However,
multiple expensive SVDs on large dense matrices are required \cite{liu2013tensor,tomioka2010estimation}.
To alleviate this problem,
we demonstrate that a similar ``sparse plus low-rank'' structure also exists 
when the scaled latent nuclear norm 
\cite{tomioka2010estimation,wimalawarne2014multitask}
is used as the regularizer.
We extend
the proposed matrix-based algorithm
to this tensor scenario. 
The resulting algorithm has low iteration cost and fast $O(1/T^2)$ convergence rate.
Experiments on matrix/tensor completion problems with both synthetic and real-world data sets
show that the proposed algorithm outperforms state-of-the-art algorithms.


Preliminary results of this paper have been reported in a shorter conference version \cite{yao2015accelerated}. 
While only the square loss is used in \cite{yao2015accelerated},
here we consider more general smooth convex loss functions.
Moreover, we extend the proposed algorithm to tensor completion and nonconvex low-rank regularization.
Besides,
post-processing is proposed to boost the recovery performance for matrix/tensor completion
with nuclear norm regularization.
All proofs can be found in Appendix~\ref{app:proof}.



\textbf{Notation.}
In the sequel, the transpose of vector/
matrix is denoted by the superscript $\cdot^\top$,
and tensors are 
denoted by boldface Euler.
For a vector $x$,
$\|x\|_1 = \sum_i |x_i|$ is its $\ell_1$-norm, 
and $\|x\| = \sqrt{\sum_i x_i^2}$ its $\ell_2$-norm.
For a matrix $X$,
$\sigma_1(X) \geq \sigma_2(X) \geq \dots \sigma_m(X)$ are  its singular values,
$\tr{X}=\sum_i X_{ii}$ is its trace, 
$\|X\|_1 = \sum_{i,j} |X_{ij}|$,
$\NM{X}{\infty}$ is its maximum singular value,
and $\|X\|_F = \tr{X^{\top}X}$ the Frobenius norm,
$\NM{X}{*} = \sum_i \sigma_i(X)$  
the nuclear norm, and
$\Span{X}$ is the column span of $X$.
Moreover, $I$ denotes the identity matrix. 

For tensors,
we follow the 
notations in \cite{kolda2009tensor}.
For a $D$-order tensor $ \tX \in \R^{I_1\times I_2 \times \cdots \times I_D}$,
its $(i_1, i_2, \dots, i_D)$th entry 
is $x_{i_1i_2\dots i_D}$.
Let $I_{D\backslash d} = \prod_{j= 1, j\neq d}^D I_j$,
the mode-$d$ matricizations 
$\tX_{\left\langle d \right\rangle }$
of $\tX$ 
is a $I_d \times I_{D\backslash d}$ matrix  
with $(\tX_{\left\langle  d \right\rangle })_{i_dj} = x_{i_1i_2\cdots i_D}$, and
$j = 1 + \sum_{l=1,l\neq d}^D(i_l - 1)\prod_{m=1, m\neq d}^{l-1}I_m$. 
Given a matrix $A$, its mode-$d$
tensorization $A^{\left\langle d \right\rangle }$ is a tensor $\tX$ with elements $x_{i_1i_2\cdots i_D} =
a_{i_dj}$, and $j$ is as defined above.
The inner product of two tensors $\tX$ and $\tY$ is $\ip{\tX , \tY} = \sum_{i_1=1}^{I_1}
\dots \sum_{i_D = 1}^{I_D} x_{i_1i_2\dots i_D}y_{i_1i_2\dots i_D}$,
and the Frobenius norm of $\tX$ is $\|\tX\|_F = \sqrt{\langle \tX, \tX \rangle}$.


For a convex but nonsmooth function $f$,
the subgradient 
is 
$g \in \partial f(x)$ where $\partial f(x) = \{u : f(y) \geq f(x) + u^{\top}\left( y - x
\right), \forall y\}$ is its subdifferential.
When $f$ is differentiable, we use $\nabla f$ for its gradient.

\section{Related Work}
\label{sec:rel}


\subsection{Proximal Algorithms}
\label{sec:prox}

Consider minimizing
composite functions of the form:
\begin{equation} 
\label{eq:F}
F(x) \equiv f(x) + g(x), 
\end{equation} 
where $f, g$ are convex,
and $f$ is smooth but $g$ is possibly non-smooth.
The proximal algorithm
\cite{parikh2014proximal}
generates a sequence of estimates $\{x_t\}$ as
\begin{equation*} 
x_{t + 1} = \prox{\mu g}{z_t}
\equiv 
\arg\min_x \frac{1}{2} \NM{x - z_t}{2}^2 + \mu g(x), 
\end{equation*} 
where
\begin{equation} \label{eq:z}
z_t = x_t - \mu\nabla f(x_t),
\end{equation} 
and
$\prox{\mu g}{\cdot}$
is the proximal operator.
When $f$ is $\rho$-Lipschitz smooth (i.e., $\|\nabla f(x_1) - \nabla
	f(x_2)\| \le \rho \|x_1 - x_2\|$)
and a fixed stepsize 
\begin{equation} \label{eq:stepsize}
\mu\leq 1/\rho
\end{equation} 
is used,
the proximal algorithm converges to the optimal solution
with a rate of $O(1/T)$, where $T$ is the number of
iterations
\cite{parikh2014proximal}.
By replacing the update in \eqref{eq:z} with
\begin{eqnarray} 
y_t & = & (1+ \theta_t) x_t - \theta_t x_{t-1}, \label{eq:y}\\
z_t & = & y_t - \mu\nabla f(y_t), \label{eq:z2}
\end{eqnarray} 
where
$\theta_{t + 1} = \frac{t - 1}{t + 2}$,
it can be accelerated  to 
a 
convergence
rate of $O(1/T^2)$ 
\cite{beck2009fast}.


Often, $g$ is ``simple'' 
in the sense that
$\prox{\mu g}{\cdot}$ can be easily obtained.
However, in more complicated problems such as overlapping group lasso \cite{jacob2009group},
$\prox{\mu g}{\cdot}$
may be expensive to compute.
To alleviate this problem, 
inexact proximal algorithm is proposed 
which allows two types of errors in standard/accelerated proximal algorithms \cite{schmidt2011convergence}: (i) an error $e_t$ in
computing $\nabla f(\cdot)$, and 
(ii) an error $\varepsilon_t$ in 
the proximal  step, i.e.,
\begin{equation} \label{eq:err}
h_{\mu g}(x_{t+1}; z_t) \le \varepsilon_t + h_{\mu g}(\prox{\mu g}{z_t}; z_t),
\end{equation} 
where 
\begin{equation} \label{eq:h}
h_{\mu g}(x;z_t) \equiv \frac{1}{2} \|x - z_t\|^2 + \mu g(x)
\end{equation} 
is the 
proximal  step's
objective.
Let the dual problem of $\min_x h_{\mu g}(x;z_t)$ be $\max_w \mathcal{D}_{\mu g}(w)$ where $w$ is the dual variable.
The the duality gap is defined as 
$\vartheta_t \equiv h_{\mu g}(x_{t + 1} ; z_t) - \mathcal{D}(w_{t + 1})$ where $w_{t + 1}$ is the corresponding dual variable of $x_{t + 1}$.
Then $\varepsilon_t$ is upper-bounded  by the duality gap $\vartheta_t$.
Thus, \eqref{eq:err} can be ensured by monitoring
$\vartheta_t$.
The following Proposition shows that by decreasing $e_t$ and $\varepsilon_t$ sufficiently fast, the convergence  rate remains
at $O(1/T^2)$. 

\begin{proposition} [\cite{schmidt2011convergence}]
	\label{cor:iapg:require}
	If $\|e_t\|$ and $\sqrt{\varepsilon_t}$ decrease as $O(1/t^{2+\delta})$ for some $\delta > 0$,
	the inexact accelerated proximal gradient algorithm converges with a rate of $O(1/T^2)$.
\end{proposition}

In the sequel, as our focus is on matrix completion, the variable $x$ in \eqref{eq:F} will be a matrix $X$.


\subsection{Soft-Impute}
\label{sec:softimpute}

Soft-Impute
\cite{mazumder2010spectral}
is a state-of-the-art algorithm for matrix completion.
At iteration $t$, let the current iterate be $X_t$.
The missing values in $O$ are filled in as
\begin{equation} \label{eq:zt}
Z_t = \SO{O} + \SOO{X_t} = \SO{O - X_t} + X_t,
\end{equation}
where 
$\Omega^{\bot}_{ij} = 1 - \Omega_{ij}$
is the complement of $\Omega$.
The next estimate $X_{t+1}$
is then generated by
the singular value thresholding (SVT) operator
\cite{cai2010singular} 
\begin{equation} \label{eq:svt}
X_{t+1} = \svt_{\l}(Z_t)
\equiv \arg \min_{X} \frac{1}{2}\NM{X - Z_t}{F}^2 + \lambda \NM{X}{*},
\end{equation}
which can be computed as follows.

\begin{lemma}[\cite{cai2010singular}] \label{lem:svt}
Let the SVD of a matrix $Z_t$ be $U \Sigma V^{\top}$.  Then,
$\svt_{\l}(Z_t) \equiv U ( \Sigma - \lambda I)_{+} V^{\top}$
where $[(A)_{+}]_{ij}=\max(A_{ij}, 0)$. 
\end{lemma}


Let $\bar{k}_t$ be the number of singular values in $Z_t$ that are larger than $\lambda$.
From Lemma~\ref{lem:svt},
a rank-$k_t$ SVD, where $k_t \ge \bar{k}_t$,
is sufficient for computing $X_{t+1}$ in (\ref{eq:svt}).
In \cite{mazumder2010spectral},
this 
rank-$k_t$ SVD
is obtained by the PROPACK algorithm \cite{larsen1998lanczos}.
The most expensive steps in computing the SVD are 
matrix-vector multiplications of the 
form $Zu$ and 
$v^{\top}Z$, 
where $u\in\R^n$ and $v \in \R^{m}$.
In general,
the above multiplications take $O( m n )$ time 
and rank-$k_t$ SVD on $Z_t$ takes $O( m n k_t )$ time.

However,
to make Soft-Impute efficient, an important observation 
in \cite{mazumder2010spectral}
is that  
$Z_t$  in (\ref{eq:zt})
has a special 
``sparse plus low-rank'' 
structure, namely that
$\SO{O - X_t}$ is sparse and $X_t$ is low-rank.
Multiplications of the form $Z_tu$ and $v^{\top}Z_t$ can 
then
be efficiently performed as follows.
Let the rank of $X_t$  be $r_t$, and its SVD be $U_t \Sigma_t {V_t}^{\top}$.
$Z_t v$ can be computed as
\begin{equation} 
\label{eq:zu}
Z_t v = \SO{O - X_t } v
+ U_t \Sigma_t ({V_t}^{\top} v).
\end{equation}
Constructing $\SO{O - X_t}$ takes $O(r_t \NM{\O}{1})$ time, while computing the products $\SO{O - X_t } u$ 
and $U_t \Sigma_t ({V_t}^{\top} u)$ take $O(\|\O\|_1)$ 
and $O(m r_t)$ time, respectively.
Similarly, $u^{\top} Z_t$ can be computed as
$u^{\top} \SO{O - X_t } + (u^{\top} U_t) \Sigma_t {V_t}^{\top}$.
Thus, to obtain the rank-$k$ SVD of $Z_t$, 
Soft-Impute needs only 
\begin{align}
O(k_t \NM{\O}{1} + r_t k_t m )
\label{eq:timesi-svd}
\end{align}
time,
and one iteration costs
\begin{align}
O((r_t + k_t) \NM{\O}{1} + r_t k_t m )
\label{eq:time-si}
\end{align}
time.
Since 
the solution is 
low-rank,
$k_t, r_t \ll m$,
and \eqref{eq:time-si} is much faster than 
the $O(m n k_t)$ time
for direct rank-$k_t$ SVD.


\vspace{-15px}

\section{Accelerated Inexact Soft-Impute}
\label{sec:alg}


In this section,
we describe the proposed 
matrix completion
algorithm.
Tibshirani \cite{tibshirani2010} suggested that acceleration is not useful for Soft-Impute,
	as it destroys the essential ``sparse plus low-rank'' structure.
However, we will show that it can indeed be preserved with acceleration.
We also show that further speedup can be achieved by using approximate SVT.

\vspace{-5px}

\subsection{Soft-Impute as a Proximal Algorithm}
\label{sec:ispg}

In (\ref{eq:mc}),
let 
\begin{equation} \label{eq:f}
f(X) = \frac{1}{2}\NM{\SO{X - O}}{F}^2= \sum_{(i,j) \in \Omega}\ell(X_{ij}, O_{ij}),
\end{equation} 
where
$\ell$ is the 
loss  function,
and $g(X) = \lambda \NM{X}{*}$.
The proximal step in the (unaccelerated) proximal algorithm is
\begin{equation*}
X_{t + 1} = \prox{\mu g}{Z_t} \equiv \arg\min_x \frac{1}{2} \NM{X - Z_t}{F}^2 + \mu \lambda \NM{X}{*},
\end{equation*}
where $Z_t = X_t - \mu \SO{X_t - O}$.
Note that the square loss 
$\ell(X_{ij}, O_{ij}) \equiv \frac{1}{2}(X_{ij} - O_{ij})^2$
in \eqref{eq:mc}
is $1$-Lipschitz smooth.
The following shows that
$f$ 
in (\ref{eq:f})
is also $1$-Lipschitz smooth.

\begin{proposition} \label{pr:lipmac}
If $\ell$ is $\rho$-Lipschitz smooth,
$f$ in (\ref{eq:f}) is also $\rho$-Lipschitz smooth.
\end{proposition}

From \eqref{eq:stepsize},
one can thus simply set $\mu = 1$ for \eqref{eq:mc}. 
We then have
$X_{t + 1} = \prox{g}{Z_t} = \svt_{\l}(Z_t)$
which is the same as \eqref{eq:svt}.
Hence, 
interestingly, 
Soft-Impute is a proximal algorithm
\cite{tibshirani2010}, 
and thus converges at a rate of $O(1/T)$ \cite{mazumder2010spectral}.


\subsection{Accelerating Soft-Impute}
\label{sec:compSVT}

Since Soft-Impute is a proximal algorithm, it is natural to 
use acceleration (Section~\ref{sec:prox}). 
Recall that the
efficiency of Soft-Impute hinges on the ``sparse plus low-rank'' structure of $Z_t$,
which allows matrix-vector multiplications of the form $Z_tu$ and $v^{\top}Z_t$ to be computed inexpensively. 
To accelerate Soft-Impute, 
from \eqref{eq:y} and  \eqref{eq:z2},
we have to compute 
\begin{equation} \label{eq:svtacc}
\prox{g}{\tZ_t} 
\! = \!\svt_\l(\tZ_t)
\! = \! \arg\min_{X} \frac{1}{2}\NM{X - \tZ_t}{F}^2 \! + \!\lambda \NM{X}{*},
\end{equation}
where
$Y_t = (1 + \theta_t) X_t - \theta_t X_{t - 1}$, and
\begin{align} 
\tZ_t = \SO{O - Y_t} + (1+ \theta_t) X_t - \theta_t X_{t-1}.
\label{eq:newz}
\end{align} 
In the following, we show that $\tZ_t$ also has a similar 
``sparse plus low-rank'' structure.

Assume that $X_t$ and $X_{t-1}$ have ranks $r_t$ and $r_{t - 1}$,  
and their SVDs are $U_t \Sigma_t V_t^{\top}$ and $U_{t-1} \Sigma_{t-1} V_{t-1}^{\top}$, respectively. 
Note that $\SO{O - Y_t}$ is sparse,
and $(1+ \theta_t) X_t - \theta_t X_{t-1}$ 
has rank at most $r_t + r_{t - 1}$.
Similar to (\ref{eq:zu}), for any $v \in \R^n$,  
we have
\begin{align*}
\tZ_t v 
\! = \!  
\SO{O \! - \! Y_t } v
\! + \! (1 \! + \! \theta_t) U_t \Sigma_t ( V_t^{\top} v)
\! - \! \theta_t U_{t-1} \Sigma_{t-1} ( V_{t-1}^{\top} v ) .
\end{align*}
The first term takes $O(\NM{\Omega}{1})$ time while the last two terms
take $O((r_{t-1} + r_t) m)$ time, 
thus
a total of
$O(\NM{\O}{1} + (r_{t-1} + r_t)m)$ time.
Similarly, for 
any $u \in \R^m$,
we have
\begin{align*} 
\!\!\!
u{\!}^{\top} \! \tZ_t
\! = \! u {\!}^{\top} \!\! \SO{O \! - \! Y_t } 
\! + \! (1 \! + \! \theta_t) (u{\!}^{\top} \! U_t) \Sigma_t V_t^{\top}
\!\!\! - \! \theta_t (u {\!}^{\top} {\!} U_{t-1}) \Sigma_{t-1} V_{t-1}^{\top}.
\end{align*}
This takes $O(\NM{\O}{1} + (r_{t-1} + r_t)m)$ time.
The rank-$k_t$ SVD of $\tZ_t$ can be obtained using PROPACK in  
\begin{align}
O( k_t \NM{\O}{1} + (r_{t-1} + r_t)k_t m)
\label{eq:timeexact}
\end{align}
time.
As the target matrix
is low-rank, $r_{t - 1}$ and $r_t$ are much smaller than $n$.
Hence, \eqref{eq:timeexact} is much faster than the $O(m n k_t)$ time required for a direct rank-$k_t$ SVD.


The accelerated algorithm has a slightly higher iteration complexity than the
unaccelerated one in \eqref{eq:timesi-svd}.
However, this is more than compensated by improvement in the convergence rate
(from $O(1/T)$ to $O(1/T^2)$),
as will be empirically demonstrated in Section~\ref{sec:matcomp:syn}.


\subsection{Approximating the SVT}
\label{sec:apprSVT}

Though acceleration preserves the ``sparse plus low-rank'' structure,
the proposed algorithm (and
Soft-Impute)
can still be computationally expensive as the SVT in each iteration uses exact 
SVD.
In this section, we show that further speedup is possible by using inexact 
SVD.

As SVT 
in \eqref{eq:svt} 
can be seen as a proximal step,
one might want to perform inexact SVT 
by monitoring the duality gap as in
Section~\ref{sec:prox}.
It can be shown that the dual of \eqref{eq:svtacc} is 
\begin{align}
\max_{\NM{W}{\infty} \le 1} \tr{W^{\top} \tZ_t} - \frac{\lambda}{2} \NM{W}{F}^2,
\label{eq:dualSVT}
\end{align}
where $W \in \R^{m \times n}$ is the dual variable.

\begin{proposition}[\cite{parikh2014proximal}] \label{lem:dualopt}
Let the SVD of matrix $\tZ_t$ be $U \Sigma V^{\top}$.  The optimal solution of
\eqref{eq:dualSVT} is $W_* = U \min(\Sigma, \lambda I) V^{\top}$, where 
$[\min(A, B)]_{ij}
	=\min(A_{ij}, B_{ij})$.
\end{proposition}

Proposition~\ref{lem:dualopt} shows that
a full SVD 
is required.
This takes $O(m^2 n)$ time and is even more expensive
than 
directly using SVT
($O(m n k_t)$ time).
Instead, the proposed approximation is motivated by the following Proposition.

\begin{proposition} \label{pr:approGSVT} 
Let $\breve{k}_t$ be the number of singular values in $\tZ_t \in \R^{m \times n}$ larger than $\lambda$,
and $Q \in \R^{m \times k_t}$, where $k_t \ge \breve{k}_t$, be orthogonal and contains
the subspace spanned by the top
$\breve{k}_t$ left singular vectors of $\tZ_t$.
Then, $\svt_\l(\tZ_t) = Q \svt_\l(Q^{\top} \tZ_t)$.
\end{proposition}

Since a low-rank solution is desired,
$k_t$ can be much smaller than $m$
\cite{mazumder2010spectral}.
Thus,
once we identify the span of $\tZ_t$'s top left singular vectors, 
we only need to perform 
SVT on 
the much smaller $Q^{\top} \tZ_t \in \R^{k_t \times n}$
(instead of $\tZ \in \R^{m \times n}$).
The question is how to find $Q$.
We adopt the power method (Algorithm~\ref{alg:powermethod}) \cite{halko2011finding},
which is more efficient 
than PROPACK
\cite{wu2000thick}.
Matrix $R_t$ in 
Algorithm~\ref{alg:powermethod}
is for warm-start.

\begin{algorithm}[ht]
\caption{$\text{PowerMethod}(\tZ_t, R_t, J)$ \cite{halko2011finding}}
	\begin{algorithmic}[1]
		\REQUIRE $\tZ_t \in \R^{m \times n}$, $R_t \in \R^{n \times k_t}$, and the 
		number of iterations $J$;
		\STATE initialize $Q_0 = \text{QR} (\tZ_t R_t)$; 
		$//$ QR$(\cdot)$ is QR factorization
		\FOR{$j = 1, 2, \dots , J$}
		\STATE $Q_j = \text{QR} (\tZ_t (\tZ_t^{\top} Q_{j - 1}))$;
		\ENDFOR
		\RETURN $Q_J$.
	\end{algorithmic}
	\label{alg:powermethod}
\end{algorithm}

Algorithm~\ref{alg:apprSVT} shows the approximate SVT procedure.
Step~1 approximates the top $k_t$ left singular vectors of $\tZ_t$ with $Q$.
In steps~2 to 5, a much smaller and less expensive  (exact) SVT is performed on
$Q^{\top}\tZ_t$.
Finally,  
$\svt_\l(\tZ_t)$ is recovered as $\tilde{X} = (Q U) \Sigma V^{\top}$ using Proposition~\ref{pr:approGSVT}.

\begin{algorithm}[ht]
\caption{Approximating the SVT of $\tZ_t$: 
	 approx-SVT$(\tZ_t, R_t, \lambda, J)$}
\begin{algorithmic}[1]
	\REQUIRE $\tZ_t \in \R^{m \times n}$, $R_t \in \R^{n \times k_t}$ and $\lambda \ge 0$;
	\STATE $Q = \text{PowerMethod}(\tZ_t, R_t, J)$;
	\STATE $[U, \Sigma, V] = \text{SVD}(Q^{\top}\tZ_t)$;
	\STATE $U = \{u_i \;| \;\sigma_i > \lambda \}$;
	\STATE $V = \{v_i \;| \;\sigma_i > \lambda \}$;
	\STATE $\Sigma = (\Sigma - \lambda I)_+$;
	\RETURN $Q U, \Sigma$ and $V$. $//$ $\tilde{X} = (Q U) \Sigma V$
\end{algorithmic}
\label{alg:apprSVT}
\end{algorithm}




\subsection{The Proposed Algorithm}
\label{sec:proalg}


We extend
problem~(\ref{eq:mc})  by
allowing
the loss $\ell$
to be 
$\rho$-Lipschitz smooth 
(e.g.,
logistic loss and squared hinge loss):
\begin{align}
\min_X F(X) \equiv
\sum_{(i,j)\in \Omega} \ell(X_{ij}, O_{ij}) + \lambda \NM{X}{*}.
\label{eq:mcgen}
\end{align}
Using \eqref{eq:z2}, 
$\tZ_t 
= Y_t - \mu \nabla f(Y_t)
= Y_t - \mu S_t$,
where $S_t$ is a sparse matrix with
\begin{equation} \label{eq:S}
[S_t]_{ij} = 
\begin{cases}
\frac{d \ell((Y_t)_{ij}, O_{ij}) }{d (Y_t)_{ij}}
& \text{if} \; (i,j) \in \Omega
\\
0
& \; \text{otherwise}
\end{cases}.
\end{equation} 
Using Proposition~\ref{pr:lipmac} and (\ref{eq:stepsize}), 
the stepsize 
$\mu$
can be set as 
$1/\rho$.
The whole procedure is shown in Algorithm~\ref{alg:AISimpute}.
The core steps 
are 6--8, which
performs
approximate SVT.
As in \cite{hsieh2014nuclear},
	$R_t$ is warm-started as $\text{QR}([V_t, V_{t-1}])$ at step~7.
Moroever,
as in \cite{o2012adaptive}, 
we restart
the algorithm 
if $F(X)$ starts to increase
(step~10).
For further speedup, $\lambda$ is dynamically reduced (step~3) by a continuation
strategy \cite{toh2010accelerated,mazumder2010spectral}.

\begin{algorithm}[ht]
	\caption{Accelerated Inexact Soft-Impute (AIS-Impute).}
	\begin{algorithmic}[1]
		\REQUIRE partially observed matrix $O$, parameter $\lambda$.
		\STATE initialize $c = 1$, $X_0 = X_1 = 0$, stepsize $\mu = 1 / \rho$,
		$\hat{\lambda} > \lambda$ and $\nu \in (0, 1)$;
		\FOR{$t = 1, 2,\dots, T $}
		\STATE $\lambda_t = (\hat{\lambda} - \lambda) \nu^{t - 1} + \lambda$;
		\STATE $Y_t = X_t + \theta_t (X_t - X_{t-1})$, where $\theta_t = \frac{c - 1}{c + 2}$;
		\STATE $\tZ_t = Y_t - \mu S_t$, with $S_t$ in (\ref{eq:S});
		\STATE $V_{t-1} = V_{t-1} - V_t({V_t}^{\top} V_{t-1})$, remove zero columns;
		\STATE $R_t = \text{QR}([V_t, V_{t-1}])$; 
		\STATE $[U_{t+1}, \Sigma_{t+1}, V_{t+1}] = \text{approx-SVT}( \tZ_t, R_t, \mu \lambda_t, J )$; 
		\\ $//$ $X_{t + 1} = U_{t+1} \Sigma_{t+1} V_{t+1}^{\top}$
		
		\STATE \textbf{if} $F(X_{t + 1}) > F(X_{t})$ \textbf{then} $c = 1$;
		\STATE \textbf{else} $c = c + 1$; \textbf{end if}
		
		\ENDFOR
		\RETURN $U_{T + 1}$, $\Sigma_{T+1}$ and $V_{T+1}$.
	\end{algorithmic}
	\label{alg:AISimpute}
\end{algorithm}



\subsection{Convergence and Time Complexity}
\label{sec:convtime}

In the following, we will show that 
the proposed algorithm has a 
convergence rate of $O(1/T^2)$.
Let $X_{t + 1} = U_{t + 1} \Sigma_{t + 1} V_{t + 1}^{\top}$
be the output of approx-SVT at step~8.
Since it only approximates $\svt_{\mu\lambda}(\tZ_t)$, there is a difference ($\varepsilon_t$ in
\eqref{eq:err}) between the proximal objectives $h_{\mu \lambda \NM{\cdot}{*}}(X_{t + 1} ; \tZ_t)$ and
$h_{\mu \lambda \NM{\cdot}{*}}(\svt_{\mu \lambda}(\tZ_t) ; \tZ_t)$ after performing step~8, 
where $h_{\mu \lambda \NM{\cdot}{*}}(\cdot;\cdot)$ is as defined in (\ref{eq:h}).
The following shows that $\varepsilon_t$ decreases at a linear rate.


\begin{proposition}
\label{pr:inexact}
Assume that
(i) $k_t \ge \breve{k}_t$ for all $t$ and $J = t$;\footnote{In practice, we simply set $J = 3$ as in
\cite{hsieh2014nuclear}.}
(ii) $\{ F(X_t) \}$ is upper-bounded.
Then $\varepsilon_t$ decreases to zero linearly.
\end{proposition}

Using Propositions~\ref{cor:iapg:require} and \ref{pr:inexact}, 
convergence of the proposed algorithm is provided 
by the following Theorem.

\begin{theorem}
\label{the:AISImpute:conv}
The sequence $\{X_t\}$ generated from
Algorithm~\ref{alg:AISimpute} converges to the optimal solution with a $O(1/T^2)$ rate.
\end{theorem}

The basic operations in the power method are multiplications of the form $\tZ_t u$ and
$v^{\top} \tZ_t$.  The tricks in Section~\ref{sec:compSVT}
can again be used for acceleration,
and computing the approximate SVT using Algorithm~\ref{alg:apprSVT}
takes only
\begin{align}
O( k_t \NM{\O}{1} + (r_{t-1} + r_t)k_t m)
\label{eq:timeinexact}
\end{align}
time.
This is slightly more expensive than 
\eqref{eq:timesi-svd}, the time for 
performing exact SVD in Soft-Impute.  However, 
Soft-Impute is not accelerated and has slower convergence than 
Algorithm~\ref{alg:AISimpute} (Theorem~\ref{the:AISImpute:conv}).
The 
complexity 
in \eqref{eq:timeinexact} is also 
the same 
as \eqref{eq:timeexact}.
However,
as will be demonstrated in Section~\ref{sec:matcomp:syn},
approximate SVT is empirically much faster.
The cost of one 
AIS-Impute 
iteration 
is summarized in Table~\ref{tab:itertime}.
This is only slightly more expensive than \eqref{eq:time-si} for Soft-Impute.

\begin{table}[ht]
\centering
\vspace{-10px}
\caption{Iteration time complexity of Algorithm~\ref{alg:AISimpute}.}
\vspace{-10px}
\label{tab:itertime}
\begin{tabular}{ c | c }
	\hline
	steps                  & complexity \\ \hline
	5   (construct  $S_t$) & $O(r_t \NM{\Omega}{1})$                                      \\ \hline
	6,7  (warm-start)      & $O(n k_t^2)$                                                 \\ \hline
	8  (approximate SVT)   & $O(k_t \|\O\|_1 + \left( r_{t-1} + r_t\right) k_t m)$        \\ \hline\hline
	total                  & $O((r_t + k_t)\NM{\Omega}{1} + (r_{t-1} + r_t + k_t) k_t m)$ \\ \hline
\end{tabular}
\end{table}

Table~\ref{tab:algs:others} compares
Algorithm~\ref{alg:AISimpute} with some existing algorithms that
will be empirically compared in Section~\ref{sec:recsys}.
Overall, Algorithm~\ref{alg:AISimpute} enjoys fast convergence and low iteration complexity.

\begin{table*}[ht]
\centering
\vspace{-5px}
\caption{Comparison of AIS-Impute (Algorithm~\ref{alg:AISimpute}) with other algorithms.
	The algorithms in \cite{hsieh2014nuclear,mishra2013low,zhang2012accelerated}
	involve solving some optimization subproblems iteratively, 
	and 
	$T_a$ is the number of iterations used.
	Moreover, integer $T_s$ and $c_1, c_2 \in (0, 1)$ are some constants.}
\vspace{-10px}
\begin{tabular}{c | c| c|c}
\hline
                  & algorithm                                         & iteration complexity                                         &        rate        \\ \hline
matrix factorization & LMaFit \cite{wen2012solving}                      & $O(\NM{\Omega}{1} k_t + m k_t)$                              &        ---         \\ \cline{2-4}
                  & ASD \cite{tanner2016low}                          & $O(\NM{\Omega}{1} k_t + m k_t)$                              &        ---         \\ \cline{2-4}
                  & R1MP \cite{wang2015orthogonal}                    & $O(\NM{\Omega}{1} +  m k_t^2)$                               &    $O(c_1^{T})$    \\ \hline
 nuclear norm     & active subspace selection \cite{hsieh2014nuclear} & $O(\NM{\Omega}{1} k_t^2 T_a)$                                & $O(c_2^{T - T_s})$ \\ \cline{2-4}
 minimization     & boost \cite{zhang2012accelerated}                 & $O(\NM{\Omega}{1} t^2 T_a)$                                  &      $O(1/T)$      \\ \cline{2-4}
                  & Sketchy \cite{yurtsever2017sketchy}               & $O(\NM{\Omega}{1} + m k_t^2)$                                &      $O(1/T)$      \\ \cline{2-4}
                  & TR \cite{mishra2013low}                           & $O(\NM{\Omega}{1} t^2 T_a)$                                  &        ---         \\ \cline{2-4}
                  & ALT-Impute \cite{hastie2015matrix}                & $O(\NM{\Omega}{1} k_t + m k_t^2)$                            &      $O(1/T)$      \\ \cline{2-4}
                  & SSGD \cite{avron2012efficient}                    & $O(m k_t^2)$                                                 &  $O(1/\sqrt{T})$   \\ \cline{2-4}
                  & AIS-Impute                                        & $O((r_t + k_t)\NM{\Omega}{1} + (r_{t-1} + r_t + k_t) m k_t)$ &     $O(1/T^2)$     \\ \hline
\end{tabular}
\label{tab:algs:others}
\vspace{-15px}
\end{table*}


\subsection{Post-Processing}
\label{sec:postmc}

The nuclear norm penalizes all singular values equally. This
may over-penalize the more important leading 
singular values.
To alleviate this problem,
we post-process the solution as in \cite{mazumder2010spectral}.
Note that only the square loss is considered 
in \cite{mazumder2010spectral}.
Here,
any smooth convex loss can be used.

Let the rank-$k$ matrix obtained from Algorithm~\ref{alg:AISimpute} be $X=U \Sigma
V^{\top}$, where $U = [u_i]$ and $V = [v_i]$.
Let $A(\theta) = U \Diag{\theta} V^{\top}$.
We undo part of the shrinkage on the singular values by replacing
$X$ with $A(\theta_*)$, where
\begin{equation}
\theta_* 
= \arg\min_{\theta} \phi(\theta)
\equiv \sum_{(i,j)\in\Omega} \ell( A(\theta)_{ij} , O_{ij} ).
\label{eq:post:mc}
\end{equation}
When $\ell$ is the square loss,
\eqref{eq:post:mc} 
has a closed-form solution 
\cite{mazumder2010spectral}.
However,
for smooth convex $\ell$ in general,
this is not possible 
and we optimize \eqref{eq:post:mc} using L-BFGS.
The most expensive step in each
L-BFGS iteration is the computation of the gradient $\nabla \phi(\theta) \in \R^k$, where
$[\nabla \phi(\theta)]_i
= u_i^{\top} B v_i$,
$B_{ij} = 
\frac{d \ell( A(\theta)_{ij}, O_{ij} )}{d A(\theta)_{ij}}$ if
$(i,j) \in \Omega$, and 
0 
otherwise.
As $S$ is sparse, 
computing $\nabla \phi(\theta)$ only takes $O(k \NM{\Omega}{1})$ time where $k \ll n$.
Thus,
one iteration of L-BFGS takes $O(k \NM{\Omega}{1})$ time,
which is not significant compared to the $O((r_t + k_t)\NM{\Omega}{1} + (r_{t-1} + r_t + k_t) k_t m)$ 
time in each AIS-Impute iteration (Table~\ref{tab:itertime}).
Moreover, L-BFGS has superlinear convergence \cite{nocedal1999numerical}.
Empirically, it converges in fewer than 10 iterations.
These make post-processing very efficient.


\subsection{Nonconvex Regularization}

While post-processing alleviates the problem of over-penalizing singular values,
recently nonconvex regularizers have been proposed to address this problem in a more direct manner.
In this section,
we first show that the proposed algorithm can be extended for truncated nuclear norm regularization (TNNR) \cite{hu2013fast},
which is a popular nonconvex variant of the nuclear norm.
Then
we show that it can be further extended for more general nonconvex regularizers.

\noindent
\textbf{Truncated Nuclear Norm.}
The optimization problem 
for TNNR \cite{hu2013fast}
can be written as
\begin{align}
\min_{X} \frac{1}{2} \NM{ \SO{X - O} }{F}^2
+ \lambda \sum_{i = r}^{m}  \sigma_i(X),
\label{eq:tnnr}
\end{align}
where $r \in \{1, \dots , m \}$.
Using DC programming
\cite{tao2005dc},
this is rewritten 
as
\begin{align}
\min_{X} \!\!
\max_{A \in \R^{m \times r}, B \in \R^{n \times r}} 
&
\frac{1}{2} \NM{ \SO{X \! - \! O} }{F}^2
\! + \! \lambda \NM{X}{*} \! - \! \lambda \tr{A^{\top} X B }
\notag
\\
\text{s.t.}
& 
\quad
A^{\top} A = I, \;
B^{\top} B = I.
\label{eq:tnnr:equiv}
\end{align} 
$A,B$ and $X$ are then obtained 
using alternative minimization
as
\begin{align}
( A_{\tau + 1}, B_{\tau + 1} ) 
= & \min_{A^{\top} A = I,	B^{\top} B = I} \tr{A^{\top} X_{\tau} B },
\label{eq:tnnr:1}
\\
X_{\tau + 1} 
= & \min_{X} \frac{1}{2} \NM{ \SO{X - O} }{F}^2 + \lambda \NM{X}{*}
\label{eq:tnnr:2}
\\
& \qquad - \lambda \tr{A_{\tau + 1}^{\top} X B_{\tau + 1} }.
\notag
\end{align}
Subproblem~\eqref{eq:tnnr:1}
has
the closed-form solution 
$A_{\tau + 1} = U_{\tau}$ and $B_{\tau + 1} = V_{\tau}$ \cite{hu2013fast},
where $U_{\tau} \Sigma_{\tau} V_{\tau}^{\top}$
is the rank-$r$ SVD of $X_{\tau}$.
Subproblem 
\eqref{eq:tnnr:2} involves convex optimization 
with the nuclear norm regularizer,
and is
solved by the accelerated proximal gradient (APG) algorithm \cite{beck2009fast} 
in \cite{hu2013fast}.

The proposed AIS-Impute can be used to solve \eqref{eq:tnnr:2} more efficiently.
Let $X_{t - 1}$ and $X_t$ be two consecutive iterates from
AIS-Impute.
As in Section~3.2,
in order to generate $X_{t + 1}$,
we compute
\begin{align}
Y_t 
& = ( 1 + \theta_t ) X_t - \theta_t X_{t - 1},
\label{eq:defy}
\\
\breve{Z}_t
& = Y_t
+ \mu \lambda A_{\tau + 1} B_{\tau + 1}^{\top}
- \mu \SO{Y_t - O}.
\notag
\end{align}
Note that 
$Y_t + \mu \lambda A_{\tau + 1} B_{\tau + 1}^{\top}$ is low-rank and $\mu \SO{Y_t - O}$ is sparse.
Thus, $\breve{Z}_t$ again has the special ``sparse plus low-rank'' structure which is key to AIS-Impute.
Each AIS-Impute iteration then takes $O((r_t + k_t)\NM{\Omega}{1} + (r_{t-1} + r_t + k_t) k_t m)$ time,
which is much cheaper than the $O(m n k)$ time for APG.

\noindent
\textbf{Other Nonconvex Low-rank Regularizers.}
Assume that
the regularizer
is of the form $r(X) = \sum_{i = 1}^m \bar{r}\left( \sigma_i ( X ) \right)$,
where $\bar{r}(\alpha)$ is a concave and nondecreasing function on $\alpha \ge 0$.
This assumption is satisfied by 
the log-sum-penalty \cite{candes2008enhancing},
minimax concave penalty \cite{zhang2010nearly},
and capped-$\ell_1$ norm \cite{zhang2010analysis}.
The corresponding optimization problem is
$\min_{X}
\frac{1}{2} \NM{\SO{X - O}}{F}^2
+ \lambda r(X)$.
As in \cite{hu2013fast},
using DC programming,
we obtain
\begin{align}
\!\!
X_{\tau + 1} & = \min_{X} \frac{1}{2} \NM{ \SO{X \! - \! O} }{F}^2 
\! + \! \lambda \sum_{i = 1}^{m} ( w_{\tau + 1} )_i \sigma_i(X), \label{eq:irnn:2}
\\
\!\!
\left[w_{\tau + 1}\right]_i
& = \hat{\partial}  \bar{r} \left( \sigma_i (X_{\tau}) \right), \;\;
i = 1, \dots, m,
\label{eq:irnn:1}
\end{align}
where $\hat{\partial} \bar{r}$ is the super-gradient 
\cite{boyd2009convex}
of $\bar{r}$.
Subproblem~\eqref{eq:irnn:1} can be easily computed in $O(m)$ time.
As for \eqref{eq:irnn:2},
its regularizer is a weighted nuclear norm.
As $\bar{r}$ is non-decreasing and concave,
$\left( w_{\tau + 1} \right)_1 \le \left( w_{\tau + 1} \right)_2 \le \cdots \le
\left( w_{\tau + 1} \right)_{m}$
\cite{lu2016nonconvex}.
The following Lemma shows that the proximal step  
in (\ref{eq:irnn:2})
has a closed-form solution.

\begin{lemma}[\cite{lu2016nonconvex,gu2017weighted}] \label{lem:irnn}
	Let the SVD of $Z$ be  $U \Sigma V^{\top}$ and 
	$0 \le w_1 \le w_2 \le
	\cdots \le w_m$. 
	The solution of
	the proximal step $\min_{X} \frac{1}{2}\NM{X - Z}{F}^2 + \lambda \sum_{i = 1}^{m} w_i \sigma_i(X)$
	is given by 
	$U \left[  \Sigma - \lambda \text{Diag}(w_1, \dots, w_m) \right]_+ V^{\top}$.
\end{lemma}
Similar to the truncated nuclear norm, 
we have
$\tZ_t  =  \SO{O - Y_t} + Y_t$,
where 
$\SO{O - Y_t}$ is sparse and 
$Y_t$ (defined in \eqref{eq:defy})
is low-rank.
Hence, we
again have the special ``sparse plus low-rank'' structure.
AIS-Impute algorithm can still be used
and one iteration takes $O((r_t + k_t)\NM{\Omega}{1} + (r_{t-1} + r_t + k_t) k_t m)$ time.


\section{Tensor Completion}
\label{sec:tensor}


Complicated data objects can often be arranged as tensors.
In this section,
we extend the proposed Algorithm~\ref{alg:AISimpute} in Section~\ref{sec:alg} from matrices to tensors.

\subsection{Tensor Model}

The nuclear norm can be defined on tensors in various ways.
The following two are the most popular.

\begin{definition}[\cite{tomioka2010estimation}] 
\label{def:tensor_norm} 
For a $D$-order tensor $\tX$,
the {\em overlapped nuclear norm\/} is
$\NM{\tX}{\text{overlap}} = \sum_{d=1}^D \lambda_d \NM{\tX_{\left\langle d \right\rangle}}{*}$,
and
the {\em scaled latent nuclear norm\/} is
$\NM{\tX}{\text{scaled}}
= \min_{ \tX^{1}, \dots, \tX^{D}
	\;:\; \sum_{d=1}^D\tX^d = \tX}
\sum_{d=1}^D
\lambda_d \NM{\tX^d_{\left\langle d \right\rangle}}{*}$.
Here, $\lambda_d \ge 0$'s are hyperparameters. 
\end{definition}
The overlapped nuclear norm regularizer penalizes nuclear norms on all modes. 
When only several modes are low-rank, 
decomposition with the scaled latent nuclear norm has
better 
generalization
\cite{tomioka2010estimation,wimalawarne2014multitask}. 
In this paper, 
we focus on 
the scaled latent nuclear norm regularizer.

Given a partially observed tensor $\ten{O} \in \R^{I_1 \times \dots \times I_D}$, 
with the observed entries indicated by $\Omega \in \{0,1\}^{I_1 \times \dots \times I_D}$.
The tensor completion problem 
can be formulated as
\begin{align}
& \min_{\tX^1, \dots, \tX^D} 
F( [ \ten{X}^1_t, \dots ,\ten{X}^D_t ] ) 
\label{eq:tc} \\
& \equiv \!\!\!\!\!\!
\sum_{(i_1, \dots, i_D) \in \Omega} \ell( \sum_{d = 1}^D \tX^d_{i_1 \dots i_D}, \ten{O}_{i_1 \dots i_D} ) 
+ \sum_{d = 1}^D \lambda_d \NM{\tX^d_{\ip{d}}}{*}.   
\notag
\end{align}
The recovered tensor is $\tX = \sum_{d = 1}^D \tX^d$.
In \cite{tomioka2010estimation,liu2013tensor},
problem~\eqref{eq:tc} is solved using ADMM \cite{Boyd2011admm}.  
However,
the ADMM update involves SVD in each iteration, 
which takes $O( \prod_{d = 1}^D I_d \sum_{d = 1}^D I_d )$ time and can be expensive.


\subsection{Generalizing SVT}
\label{sec:extprox}

In \eqref{eq:tc},
let 
\begin{eqnarray}
\!\!\!\!\!\! f([\tX^1, \dots, \tX^D]) 
\!\!\! & = & \!\!\!\!\!\!\!\!\!\!\!\! \sum_{(i_1,\dots,i_D) \in \Omega} 
\ell( \sum_{d = 1}^D \tX^d_{i_1\dots i_D}, \ten{O}_{i_1\dots i_D} )\!\!,
\label{eq:ften}\\
\!\!\!\!\!\! g([\tX^1, \dots, \tX^D]) 
\!\!\! & = & \!\!\! \sum_{d = 1}^D \lambda_d \NM{\tX^d_{\ip{d}}}{*}.
\label{eq:gten}
\end{eqnarray}
The iterates in Algorithm~\ref{alg:AISimpute} are generated by SVT.  
As there are 
multiple nuclear norms in 
(\ref{eq:gten}),
the following
extends SVT  for this case.

As $g$ in \eqref{eq:gten} is separable w.r.t. 
$\tX^i$'s,
one can compute the proximal step for each 
$\tX^i$
separately \cite{parikh2014proximal}.
Updates (\ref{eq:y}), (\ref{eq:z2}) 
in the APG 
become
\begin{eqnarray}
\ten{Y}_t^d & = & (1 + \theta_t) \tX^d_t - \theta_t \tX^d_{t - 1}, \nonumber\\
\breve{\ten{Z}}^d_t & = & \ten{Y}_t^d 
- \mu \ten{S}_t
= (1 + \theta_t) \tX^d_t - \theta_t \tX^d_{t - 1} - \mu \ten{S}_t,
\label{eq:tcacc}
\end{eqnarray}
for $d=1, \dots, D$,
where 
$\ten{S}_t$ is a sparse tensor with
\begin{eqnarray}
(\ten{S}_t)_{i_1 \dots i_D}
= 
\begin{cases}
\frac{d \ell((\hat{\ten{Y}}_t)_{i_1 \dots i_D}, \ten{O}_{i_1 \dots i_D} ) }
{d (\hat{\ten{Y}}_t)_{i_1 \dots i_D}}
& \text{if\;} (i_1, \dots, i_D) \in \Omega
\\
0
&
\text{otherwise}
\end{cases},
\label{eq:tcaccs}
\end{eqnarray}
and $\hat{\ten{Y}}_t = \sum_{d = 1}^D \ten{Y}_t^d$.
Lemma~\ref{lem:svt} is also extended 
to $[\tX^1_{t + 1}, \dots, \tX^D_{t + 1}] = \prox{\mu g}{[\breve{\ten{Z}}_t^1, \dots, \breve{\ten{Z}}_t^D]}$
as follows.

\begin{proposition} \label{pr:gsvt}
$( \tX^d_{t + 1} )_{\ip{d}} = \svt_{\mu\lambda_d \NM{\cdot}{*}}((\breve{\ten{Z}}_t^d)_{\ip{d}} )$.
\end{proposition}

The stepsize rule in \eqref{eq:stepsize} depends on 
the modulus of Lipschitz smoothness of $f$,
which is given by the following.

\begin{proposition} \label{pr:lipten}
If $\ell$ is $\rho$-Lipschitz smooth,
$f$ in \eqref{eq:ften} is $\sqrt{D}\rho$-Lipschitz smooth.
\end{proposition}

Proposition~\ref{pr:approGSVT} can be used to reduce the size  of
$(\breve{\ten{Z}}_t^d)_{\ip{d}}$
in 
Proposition~\ref{pr:gsvt}, and Algorithm~\ref{alg:powermethod} can be used to approximate the
underlying SVD.
However, this is still not fast enough.
Assume that $\breve{k}^d_t$ singular values in $(\breve{\ten{Z}}_t^d)_{\ip{d}}$ are larger than $\mu\lambda_d$,
and rank-$k^d_t$ SVD,
where $k^d_t \ge \breve{k}^d_t$, 
is performed.
SVT on $(\ten{Z}_t^d)_{\ip{d}}$ takes $O( k^d_t \prod_{d = 1}^D I_d )$ time.
As SVT has to be performed on each mode,
one iteration of APG takes $O( \prod_{d = 1}^{D} I_d \sum_{d = 1}^D k^d_t )$ time,
which is expensive.

\subsection{Fast Approximate SVT with Special Structure}
\label{sec:specstuc}

In Section~\ref{sec:compSVT},
the special ``sparse plus low-rank'' structure can greatly reduce the time complexity
of matrix multiplications.
As $\tX^d_{t - 1}, \tX^d_t$ are low-rank tensors and $\ten{S}_t$ is sparse, 
$\breve{\ten{Z}}^d_t$ in 
\eqref{eq:tcacc} also has the
``sparse plus low-rank'' structure.
However, 
to generate $\tX^d_{t + 1}$ using Proposition~\ref{pr:gsvt},
we need to perform matrix multiplications of the form
$(\breve{\ten{Z}}^d_t)_{\ip{d}} v$, where $ v \in \R^{I_{D \setminus d}}$,
and $u^{\top} (\breve{\ten{Z}}_t)_{\ip{d}}$, where  $u \in \R^{I_d}$.
Unfolding $\breve{\ten{Z}}_t$
takes $O( \prod_{d = 1}^{D} I_d )$ time and can be expensive.
In the following, we show how this can be avoided.

To generate $(\ten{X}_{t + 1}^d)_{\ip{d}}$,
it can be seen 
from Proposition~\ref{pr:gsvt} and \eqref{eq:tcacc}
that 
$\ten{X}_t^d$ and $\ten{X}_{t - 1}^d$ 
only need to be unfolded along their $d$th modes.
Hence, 
instead of storing them as tensors,
we store 
$(\tX^d_t)_{\ip{d}}$ as its rank-$r^d_t$  SVD $U^d_t \Sigma^d_t {V^d_t}^{\top}$, and $(\tX^d_{t - 1})_{\ip{d}}$ as its rank-$r^d_{t - 1}$ SVD $U^d_{t - 1} \Sigma^d_{t - 1} {V^d_{t - 1}}^{\top}$.
For any $ v \in \R^{I_{D \setminus d}}$,
\begin{eqnarray*}
(\breve{\ten{Z}}^d_t)_{\ip{d}} v
& = & (1 + \theta_t) U^d_t \Sigma^d_t ({V^d_t}^{\top} v) 
\notag \\
& & - \theta_t U^d_{t - 1} \Sigma^d_{t - 1} ({V^d_{t - 1}}^{\top} v) - \mu (\ten{S}_t)_{\ip{d}} v.
\label{eq:splrten1} 
\end{eqnarray*}
The first two terms 
can be computed in $O( (I_d + I_{D \setminus d})(r^d_t + r^d_{t - 1}) )$ time.
As $\ten{S}_t$ is sparse, computing
the last term takes $O(\NM{\Omega}{1})$ time.
Thus, 
$(\breve{\ten{Z}}^d_t)_{\ip{d}} v$
can be obtained in
$O( \NM{\Omega}{1} + (I_d + I_{D \setminus d})(r^d_t + r^d_{t - 1}) )$ time.
Similarly, 
for any $u \in \R^{I_d}$,
$u^{\top} (\breve{\ten{Z}}_t)_{\ip{d}}$
can be computed in
$O( \NM{\Omega}{1} + (I_d + I_{D \setminus d})(r^d_t + r^d_{t - 1}) )$ time.
Thus, performing approximate SVT on $(\breve{\ten{Z}}_t^d)_{\ip{d}}$,
with rank $k^d_t \ge \breve{k}^d_t$,
using
Algorithm~\ref{alg:apprSVT} takes 
$ O( k^d_t \NM{\Omega}{1} + k^d_t (I_d + I_{D \setminus d}) (r^d_t + r^d_{t - 1}) ) $
time.
Using Proposition~\ref{pr:gsvt},
solving the proximal step 
$\prox{\mu g}{[\breve{\ten{Z}}_t^1, \dots, \breve{\ten{Z}}_t^D]}$ takes a total of
\begin{align}
O(  \sum_{d = 1}^D k^d_t \NM{\Omega}{1} + k^d_t (I_d + I_{D \setminus d}) (r^d_t + r^d_{t - 1}) ) 
\label{eq:tentime1}
\end{align}
time.
As the target tensor is low-rank, $r^d_t, k^d_t \ll I_d$ for $d = 1, \dots, D$. Hence, 
\eqref{eq:tentime1} is much faster than 
directly using Proposition~\ref{pr:gsvt}
($O( \prod_{d = 1}^{D} I_d \sum_{d = 1}^D k^d_t )$ time).

\subsection{The Proposed Algorithm}

The whole procedure is shown in
Algorithm~\ref{alg:AISImpute:ten}.
Unlike, Algorithm~\ref{alg:AISimpute}, $D$ SVTs
have to be computed (steps~5-11) in each iteration.

\begin{algorithm}[ht]
\caption{AIS-Impute (tensor case).}
	\begin{algorithmic}[1]
		\REQUIRE partially observed tensor $\ten{O}$, parameter $\lambda$;
		\STATE initialize $c = 1$, $\tX^1_0 = \dots = \tX^D_0 = 0$, 
		$\tX^1_1 = \dots = \tX^D_1 = 0$,
		step-size $\mu = 1/(\sqrt{D}\rho)$,
		$\hat{\lambda} > \max_{d = 1, \dots, D}\lambda_d$ and $\nu \in (0, 1)$;
		\FOR{$t = 1, 2, \dots, T$}
		\STATE $\theta_t = (c - 1)/(c + 2)$;
		\STATE construct the sparse observed tensor $\ten{S}_t$ from \eqref{eq:tcaccs};
		\FOR{$d = 1, \dots, D$}
		\STATE $(\lambda_d)_t = 
		(\hat{\lambda} - \lambda_d) 
		\nu^{t - 1} 
		+ \lambda_d$;
		\STATE $\breve{\ten{Z}}_t^d = (1 + \theta_t) \tX^d_t - \theta_t \tX^d_{t - 1} + \mu \ten{S}_t$;
		\STATE $V^d_{t-1} \! = \! V^d_{t-1} \! - \! V^d_t((V^d_t)^{\top} V^d_{t-1})$, 
		remove zero columns;
		\STATE $R^d_t = \text{QR}(\left[V^d_t, V^d_{t-1} \right])$; 
		\STATE $\left[ U^d_{t+1}, \Sigma^d_{t+1}, V^d_{t+1} \right]$\\
		= $\text{approx-SVT} ( (\breve{\ten{Z}}^d_t)_{\ip{d}}, R^d_t, \mu (\lambda_d)_t, J )$; 
		\\ $//$ $\ten{X}_{\ip{d}}^{d} = U^d_{t+1} \Sigma^d_{t+1} (V^d_{t+1})^{\top}$
		\ENDFOR
		
		\STATE \textbf{if} $F( [ \ten{X}^1_{t + 1}, \dots, \ten{X}^D_{t + 1} ] )  
		\! > \! F( [ \ten{X}^1_t, \dots ,\ten{X}^D_t ] )$ \textbf{then} $c = 1$;
		\STATE \textbf{else} $c = c + 1$; \textbf{end if}
		
		\ENDFOR
		\RETURN $U^d_{t+1}$, $\Sigma^d_{t+1}$, $V^d_{t+1}$ where $d = 1, \dots, D$.
	\end{algorithmic}
	\label{alg:AISImpute:ten}
\end{algorithm}

Analogous to Theorem~\ref{the:AISImpute:conv},
we have the following.

\begin{theorem}
\label{the:AISImpute:ten:conv}
Assume that (i) $k_t^d \ge \breve{k}_t^d$ for $d = 1, \dots, D$, all $t$ and $J = t$; 
(ii) $F( [ \ten{X}^1_t, \dots ,\ten{X}^D_t ] )$ is upper bounded.
The sequence $\{ [\tX^1_t, \dots, \tX^D_t] \}$ generated from Algorithm~\ref{alg:AISImpute:ten}
converges to the optimal solution with
a $O(1/T^2)$ rate.
\end{theorem}


\subsection{Post-Processing}
\label{sec:postproc}

As in Section~\ref{sec:postmc},
the nuclear norm regularizer in \eqref{eq:tc}
may over-penalize
top singular values.
To undo such shrinkage and boost recovery performance,
we also adopt post-processing here.
Let the tensor output from Algorithm~\ref{alg:AISImpute:ten} be $\ten{X} = \sum_{d = 1}^D
\ten{X}^d$,
where $\ten{X}^d_{\ip{d}} = U^d \Sigma^d (V^d)^{\top}$
has rank $k^d$.
Define 
$\ten{A}( \theta^1, \dots, \theta^D ) = 
\sum_{d = 1}^D (  U^d \Diag{\theta^d} (V^d)^{\top} )_{\ip{d}}$.
As in  \eqref{eq:post:mc},
we replace $\ten{X}$ with $\ten{A}\left( \theta^1_*, \dots, \theta^D_* \right)$,
where
\begin{align}
\label{eq:post:tc}
[ ( \theta_*^1 )^{\top}, \dots, ( \theta_*^D )^{\top} ]^{\top} =
\arg\min_{\theta^1, \dots, \theta^D} \phi(\theta^1, \dots, \theta^D) ,
\end{align}
and
\begin{align*}
\phi( \theta^1, \dots, \theta^D )
= \sum_{(i_1, \dots, i_D)\in\Omega}
\ell( \ten{A}(\theta^1, \dots, \theta^D)_{i_1\dots i_D}, \ten{O}_{i_1\dots i_D} ).
\end{align*}
As \eqref{eq:post:tc} is a smooth convex problem,
L-BFGS is used for optimization.
Let $U^d = [u^d]$ and $V^d = [v^d]$.
Then, 
$\nabla \phi(\theta^1, \dots, \theta^D) = [ (w^1)^{\top}, \dots, (w^D)^{\top} ]^{\top}$
where $w^d = [w^d_i] \in k^d$,
$w^d_i = (u^d_i)^{\top}
\ten{B}_{\ip{d}} v^d_i$,
and
$\ten{B}_{i_1\dots i_D} = \frac{d \ell\left( \ten{A}([\theta^1,\dots,\theta^D])_{i_1\dots i_D}, \ten{O}_{i_1\dots i_D} \right)}{d \ten{A}([\theta^1,\dots,\theta^D])_{i_1\dots i_D}} $
if $(i_1,\dots, i_D) \in \Omega$ and $0$ otherwise.
Computation of $\nabla \phi(\theta^1, \dots, \theta^D)$ takes 
$O(\sum_{d = 1}^D k^d \NM{\Omega}{1})$ time,
which is comparable to the per-iteration complexity of AIS-Impute in \eqref{eq:tentime1}
and is very efficient.
Thus, each L-BFGS iteration is inexpensive.
As for the matrix case,
empirically,
L-BFGS converges in fewer than 10 iterations.
These make post-processing very efficient.


\section{Experiments}
\label{sec:expts}

In this section, we perform experiments on 
matrix completion (Sections~\ref{sec:matcomp:syn}-\ref{sec:explink}) and tensor
completion (Sections~\ref{sec:tensorsyn}, \ref{sec:multen}).
Experiments
are performed on a PC with Intel Xeon E5-2695 CPU and 256GB RAM.
All algorithms are implemented in Matlab.


\subsection{Synthetic Data}
\label{sec:matcomp:syn}

\begin{table*}[ht]
\centering
\vspace{-5px}
\caption{Matrix completion results on the synthetic data.
	Here, sparsity is the proportion of observed entries,
	and post-processing
	time is in seconds.}
\vspace{-10px}
\begin{tabular}{c | c | c | c |  c || c }
	\hline
	&                    &                \multicolumn{2}{c|}{NMSE}                &      & \\
	&                    & without post-processsing               & with post-processing             & rank & post-processing time      \\ \hline
	$m=250$     & APG                & \textbf{0.0167$\pm$0.0007} & \textbf{0.0098$\pm$0.0001} & 5    & 0.01      \\ \cline{2-6}
(sparsity: 33.1\%) & Soft-Impute        & \textbf{0.0166$\pm$0.0007} & 0.0099$\pm$0.0001          & 5    & 0.01      \\ \cline{2-6}
	& AIS-Impute (exact) & \textbf{0.0165$\pm$0.0007} & \textbf{0.0098$\pm$0.0001} & 5    & 0.01      \\ \cline{2-6}
	& AIS-Impute         & \textbf{0.0165$\pm$0.0007} & \textbf{0.0098$\pm$0.0001} & 5    & 0.01      \\ \hline
	$m=1000$     & APG                & \textbf{0.0165$\pm$0.0001} & \textbf{0.0090$\pm$0.0001} & 5    & 0.01      \\ \cline{2-6}
	(sparsity: 10.4\%) & Soft-Impute        & 0.0170$\pm$0.0005          & 0.0097$\pm$0.0001          & 5    & 0.03      \\ \cline{2-6}
	& AIS-Impute (exact) & \textbf{0.0166$\pm$0.0001} & 0.0093$\pm$0.0001          & 5    & 0.02      \\ \cline{2-6}
	& AIS-Impute         & \textbf{0.0166$\pm$0.0001} & 0.0092$\pm$0.0001          & 5    & 0.02      \\ \hline
	$m=4000$     & APG                & \textbf{0.0142$\pm$0.0002} & 0.0080$\pm$0.0001          & 5    & 0.05      \\ \cline{2-6}
	(sparsity: 3.1\%)  & Soft-Impute        & \textbf{0.0143$\pm$0.0003} & 0.0082$\pm$0.0002          & 5    & 0.18      \\ \cline{2-6}
	& AIS-Impute (exact) & \textbf{0.0142$\pm$0.0002} & \textbf{0.0080$\pm$0.0001} & 5    & 0.11      \\ \cline{2-6}
	& AIS-Impute         & \textbf{0.0142$\pm$0.0002} & \textbf{0.0080$\pm$0.0001} & 5    & 0.13      \\ \hline
\end{tabular}
\label{tab:matperf}
\vspace{-15px}
\end{table*}

In this section, we perform matrix completion experiments with synthetic data. 
The ground-truth 
matrix has a rank of 5, and is 
generated as $O =
U V \in
\R^{m\times m}$,  where the entries of $U \in \R^{m \times 5}$ and $V \in \R^{5 \times m}$
are sampled i.i.d. from the standard normal distribution $\mathcal{N}(0, 1)$.
Noise, sampled from $\mathcal{N}(0, 0.05)$, is then added.
We 
randomly
choose 
$15 m \log(m)$  of the 
entries
in $O$ as observed.
Half of them are used for training, and the other half as validation
set for parameter tuning.
Testing is performed on the unobserved (missing) entries.
We vary $m$ in the range $\{250, 1000, 4000\}$.

The following proximal algorithms
are compared:
(i) accelerated proximal gradient algorithm
(denoted ``APG") \cite{toh2010accelerated}: It uses PROPACK to obtain singular values that are larger than $\lambda$;
(ii) Soft-Impute
\cite{mazumder2010spectral};
(iii) AIS-Impute (the proposed Algorithm~\ref{alg:AISimpute}); and
(iv) AIS-Impute (exact): This is a variant of the proposed algorithm with
exact SVT step (computed using PROPACK).

Let $X$ be the recovered matrix.
For performance evaluation,
we use the (i) normalized mean squared error
$\text{NMSE} = \NM{\SOO{X - U V}}{F}/ \NM{\SOO{U V}}{F}$,
and
(ii) rank of $X$.
To  reduce statistical variability, experimental results are averaged over 5 repetitions. 

Results\footnote{The lowest and comparable results (according to the pairwise t-test with 95\% confidence) are highlighted.}
are shown in 
Table~\ref{tab:matperf}.
As can be seen, all algorithms have similar NMSE performance, with Soft-Impute being slightly worse.
The plots of objective value vs time and iterations are shown in Figure~\ref{fig:synmatcomp}.
In terms of the number of iterations,
the accelerated algorithms (APG, AIS-Impute(exact)
and AIS-Impute)
are very similar 
and converge much faster than Soft-Impute
(which only has a $O(1/T)$ convergence rate).
However, in terms of time,
both APG and Soft-Impute are slow, 
as APG does not utilize the ``sparse plus low-rank'' structure and Soft-Impute has slow convergence.
AIS-Impute(exact) is consistently faster than APG and Soft-Impute,
as both acceleration and ``sparse plus low-rank'' structure are utilized.
However, AIS-Impute is the fastest
as it further allows inexact updates of the proximal step.
This also verifies our motivation of using the approximate SVT in Section~\ref{sec:apprSVT}.

\begin{figure}[ht]
\centering

\subfigure[$m = 250$.]
{\includegraphics[width=0.24\textwidth]{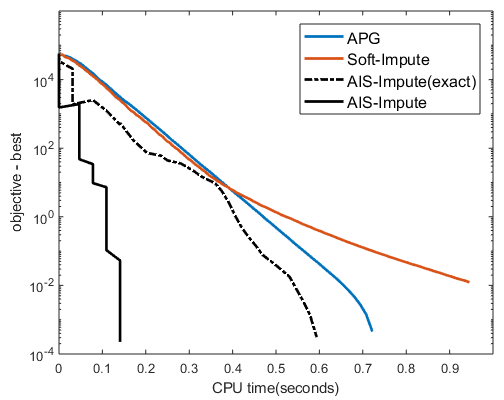}
	\includegraphics[width=0.245\textwidth]{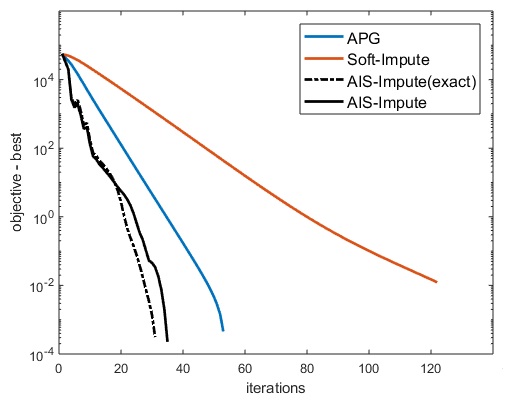}}
\subfigure[$m = 1000$.]
{\includegraphics[width=0.24\textwidth]{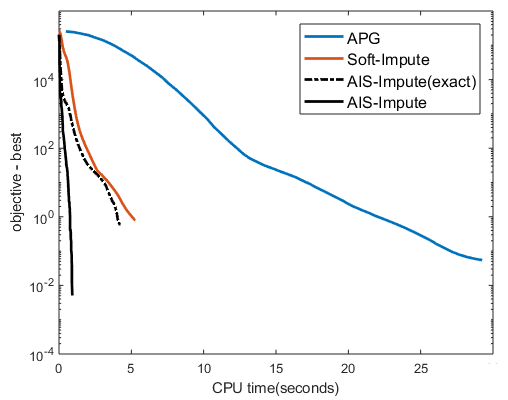}
	\includegraphics[width=0.24\textwidth]{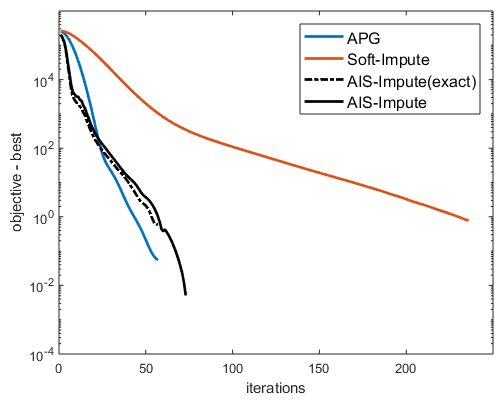}}
\subfigure[$m = 4000$.]
{\includegraphics[width=0.24\textwidth]{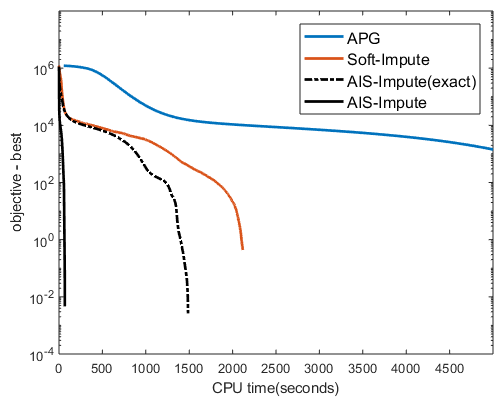}
	\includegraphics[width=0.24\textwidth]{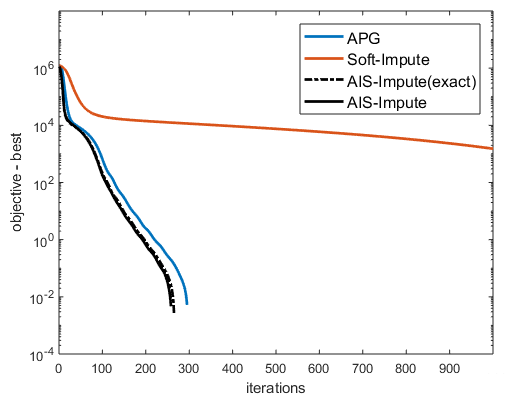}}

\vspace{-10px}
\caption{Convergence of objective value on the synthetic data.
	Left: vs CPU time (in seconds); Right: vs number of iterations
	(note that AIS-Impute(exact) and AIS-Impute overlap with each other).}
\label{fig:synmatcomp}
\vspace{-20px}
\end{figure}


Table~\ref{tab:matperf} also shows the NMSE results with post-processing (Section~\ref{sec:postmc}).
Compared to the time used by the main algorithm
(Figure~\ref{fig:synmatcomp}),
the post-processing time is small and can be ignored.
Thus, post-processing is always performed in the sequel.


\subsection{Recommender System}
\label{sec:recsys}

In this section,
experiments are performed on two well-known
benchmark data sets, 
\textit{MovieLens} and \textit{Netflix}.

\noindent
\textbf{MovieLens.}
The \textit{MovieLens} data set
(Table~\ref{tab:recSys})
contains ratings 
($\{1,2,3,4,5\}$)
of different users on movies.
It has been commonly used  in matrix completion experiments
\cite{mazumder2010spectral,hsieh2014nuclear}.
We randomly use $50\%$ of the observed ratings for training, $25\%$ for validation and the rest for testing.

\begin{table}[ht]
\centering
\vspace{-10px}
\caption{\textit{MovieLens} data sets used in the experiments.}
\vspace{-10px}
\begin{tabular}{c | c | c | c}
	\hline
	              & \#users & \#movies & \# observed ratings \\ \hline
	\textit{100K} & 943     & 1,682    & 100,000             \\ \hline
	 \textit{1M}  & 6,040   & 3,449    & 999,714             \\ \hline
	\textit{10M}  & 69,878  & 10,677   & 10,000,054          \\ \hline
\end{tabular}
\label{tab:recSys}
\vspace{-10px}
\end{table}

\begin{table*}[ht]
\centering
\vspace{-5px}
\caption{Results on the \textit{MovieLens} data sets. Note that TR and APG
	cannot converge 
	in $10^4$ seconds  on the 
		\textit{10M} data set.}
\vspace{-10px}
\begin{tabular}{c | c | c  c | c  c | c  c}
	\hline
	   \multicolumn{2}{c|}{}    & \multicolumn{2}{c|}{\textit{100K}} & \multicolumn{2}{c|}{\textit{1M}} & \multicolumn{2}{c}{\textit{10M}} \\
	   \multicolumn{2}{c|}{}    & RMSE                     & rank    & RMSE                     & rank  & RMSE                     & rank  \\ \hline
	factorization & LMaFit      & 0.896$\pm$0.011          & 3       & 0.827$\pm$0.002          & 6     & 0.819$\pm$0.001          & 12    \\ \cline{2-8}
	              & ASD         & 0.905$\pm$0.055          & 3       & 0.826$\pm$0.004          & 6     & 0.816$\pm$0.002          & 12    \\ \cline{2-8}
	              & R1MP        & 0.938$\pm$0.016          & 10      & 0.857$\pm$0.001          & 19    & 0.853$\pm$0.002          & 27    \\ \hline
	nuclear norm  & active      & \textbf{0.880$\pm$0.003} & 8       & \textbf{0.821$\pm$0.001} & 16    & \textbf{0.803$\pm$0.001} & 72    \\ \cline{2-8}
	minimization  & boost       & \textbf{0.881$\pm$0.003} & 8       & \textbf{0.821$\pm$0.001} & 16    & 0.814$\pm$0.001          & 15    \\ \cline{2-8}
	              & Sketchy     & 0.889$\pm$0.003          & 8       & \textbf{0.821$\pm$0.001} & 48    & 0.826$\pm$0.001          & 60    \\ \cline{2-8}
	              & TR          & 0.884$\pm$0.002          & 8       & \textbf{0.820$\pm$0.001} & 20    & ---                      & ---   \\ \cline{2-8}
	              & SSGD        & 0.886$\pm$0.011          & 8       & 0.849$\pm$0.006          & 16    & 0.858$\pm$0.014          & 45    \\ \cline{2-8}
	              & APG         & \textbf{0.880$\pm$0.003} & 8       & \textbf{0.820$\pm$0.001} & 16    & ---                      & ---   \\ \cline{2-8}
	              & Soft-Impute & \textbf{0.881$\pm$0.003} & 8       & \textbf{0.821$\pm$0.001} & 16    & \textbf{0.803$\pm$0.001} & 72    \\ \cline{2-8}
	              & ALT-Impute  & \textbf{0.882$\pm$0.003} & 8       & 0.823$\pm$0.001          & 16    & 0.805$\pm$0.001          & 45    \\ \cline{2-8}
	              & AIS-Impute  & \textbf{0.880$\pm$0.003} & 8       & \textbf{0.820$\pm$0.001} & 16    & \textbf{0.802$\pm$0.001} & 72    \\ \hline
\end{tabular}
\label{tab:movieLens}
\vspace{-5px}
\end{table*}

\begin{figure*}[ht]
\centering
\subfigure[\textit{100K}.]
{\includegraphics[width= 0.24 \textwidth]{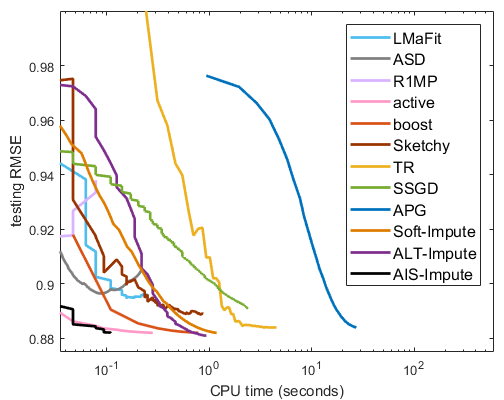}}
\qquad
\subfigure[\textit{1M}.]
{\includegraphics[width= 0.24 \textwidth]{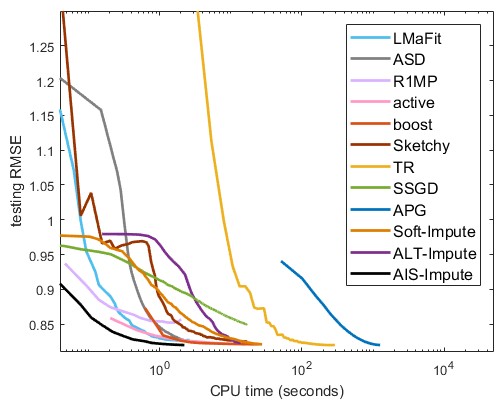}}
\qquad
\subfigure[\textit{10M}.]
{\includegraphics[width= 0.24 \textwidth]{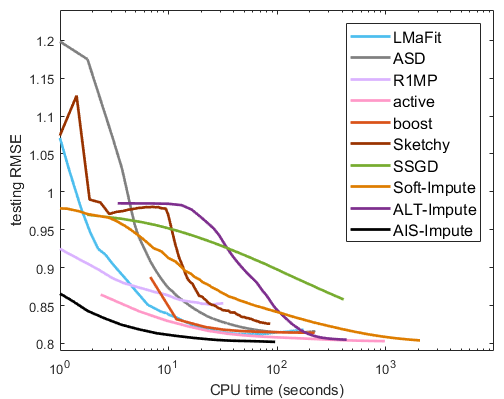}}	
\vspace{-10px}
\caption{Testing RMSE vs CPU time (in seconds) on \textit{MovieLens} data sets.}
\label{fig:movielens}
\vspace{-15px}
\end{figure*}

We compare AIS-Impute with the two most popular low-rank matrix learning approaches
\cite{koren2008factorization,candes2009exact}, namely,
factorization-based and nuclear-norm minimization methods.
The factorization-based methods include (i) large scale matrix fit (``LMaFit")
\cite{wen2012solving}, which uses alternative minimization with over-relaxation;
(ii) alternative steepest descent (``ASD'') \cite{tanner2016low}, which uses alternating steepest descent; (iii) rank-one matrix pursuit (``R1MP'') \cite{wang2015orthogonal}, which greedily pursues a rank-one basis in each iteration. 	
The nuclear-norm minimization methods include (i) active subspace selection
(``active") \cite{hsieh2014nuclear}, which uses the power method in each iteration
to identify the active row and column subspaces; (ii) a boosting approach
(``boost'') \cite{zhang2012accelerated}, which uses a variant of the Frank-Wolfe
(FW) algorithm \cite{frank1956algorithm}, with local optimization 
in each iteration
using L-BFGS; (iii) sketchy decisions (``Sketchy'') \cite{yurtsever2017sketchy}, 
which is also a FW variant, 
and uses random matrix projection \cite{halko2011finding} to reduce 
the space and per-iteration time complexities;
(iv) second-order trust-region algorithm (``TR") \cite{mishra2013low}, which alternates fixed-rank optimization and rank-one updates; 
(vi)
stochastic gradient descent (``SSGD'') \cite{avron2012efficient}, which is a stochastic gradient descent algorithm;
and
(v) matrix completion based on fast alternating least squares (``ALT-Impute'') \cite{hastie2015matrix}, 
which is a fast variant of Soft-Impute \cite{mazumder2010spectral} that avoids SVD by alternating least squares.
For all algorithms,
parameters are tuned using the validation set.
The algorithm is stopped when the relative change in 
objectives between consecutive iterations is smaller than ${10}^{-4}$.

For performance evaluation, 
as in \cite{hsieh2014nuclear,mazumder2010spectral}, 
we use (i) the 
root mean squared error 
on the test set:
$\text{RMSE} = \NM{P_{\hat{\Omega}}(X - \hat{O})}{F} / (\NM{\hat{\Omega}}{1})^\frac{1}{2}$, where
$X$ is the recovered matrix,
and the testing ratings $\{\hat{O}_{ij}\}$ is indexed by the set
$\hat{\Omega}$;
and (ii) rank  of $X$. 
The experiment is repeated 5 times and the average performance is reported.

Results are shown in Table~\ref{tab:movieLens}. As can be seen,
AIS-Impute is consistently the fastest and has the lowest RMSE.
On \textit{MovieLens-10M},
TR and APG are not run 
as they are too slow.
Figure~\ref{fig:movielens} shows the
testing RMSE 
with 
CPU time.
As can be seen,
Boost, TR, SSGD and APG are all very slow.
Boost and TR need to solve an expensive subproblem in each iteration;
SSGD has slow convergence; while
APG requires SVD and does not utilize the ``sparse plus low-rank" structure for fast matrix multiplication.
ALT-Impute and LMaFit do not need SVT,
and are faster than Soft-Impute.
However, their 
nonconvex 
formulations 
have slow convergence,
and are thus slower than AIS-Impute.
Overall, 
AIS-Impute is the fastest,
as it combines inexpensive iteration with fast convergence.

\begin{figure*}[ht]
\centering
\subfigure[$\lambda = \lambda_0/10$.]
{\includegraphics[width=0.2375\textwidth]{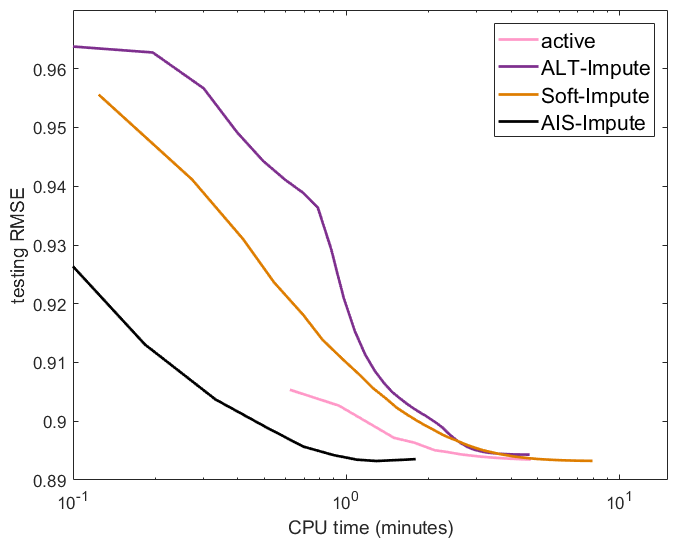}}
\qquad
\subfigure[$\lambda = \lambda_0/20$.]
{\includegraphics[width=0.245\textwidth]{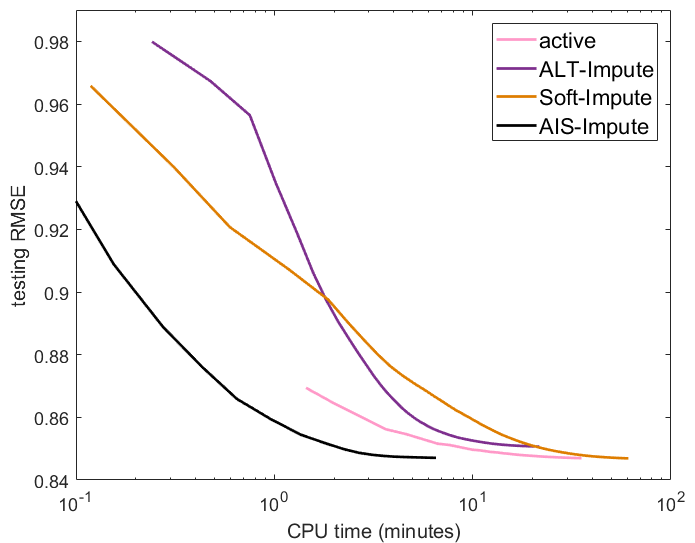}}
\qquad
\subfigure[$\lambda = \lambda_0/30$.]
{\includegraphics[width=0.24\textwidth]{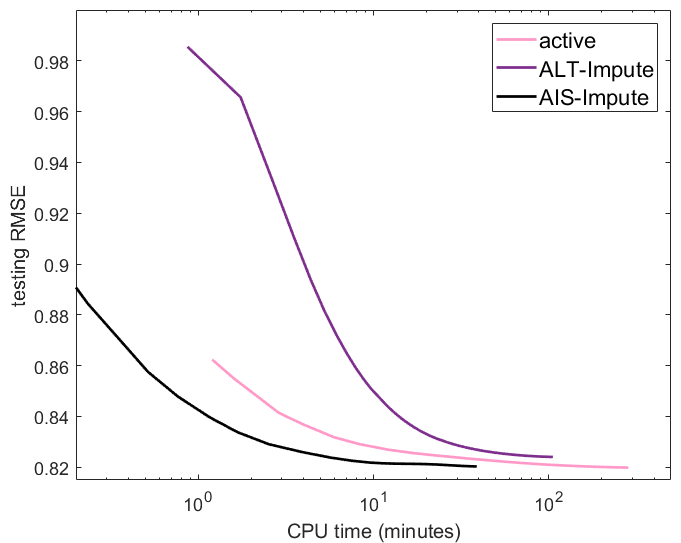}}

\vspace{-10px}
\caption{Testing RMSE vs CPU time (in minutes) 
	on the \textit{Netflix} data set, with various values for
	the regularization parameter $\lambda$.}
\label{fig:netfixRMSE}
\vspace{-15px}
\end{figure*}

\noindent
\textbf{Netflix}. 
The \textit{Netflix} data set
contains ratings of 480,189 users on 17,770 movies. 
$1\%$ of the ratings matrix are observed.
We randomly sample $50 \%$ of the observed ratings for training, and the rest for testing.

We only compare with
active subspace selection, ALT-Impute and Soft-Impute; while methods including
boost, TR, SSGD, APG are slow and not compared.
LMaFit solves 
a different 
optimization problem 
based on matrix factorization,
and has worse recovery performance than AIS-Impute. Thus,
it is also not compared.
As in \cite{mazumder2010spectral}, 
several choices of the regularization parameter $\lambda$ 
are experimented.

Results are shown in Table~\ref{tab:netflix}.
As in previous experiments, the RMSEs and ranks obtained by the various algorithms are similar.
Figure~\ref{fig:netfixRMSE} shows the 
plot of testing RMSE versus CPU time.
As can be seen, AIS-Impute is again much faster. 

\begin{table}[ht]
\centering
\vspace{-5px}
\caption{Results on the \textit{Netflix} data set.
The regularization parameter $\lambda$ in \eqref{eq:mc} is set as $\lambda_0 / c$,
where $\lambda_0 = \NM{\SO{O}}{F}$.
Soft-Impute with $c = 30$ is not run as it is very slow.}
\vspace{-10px}
\begin{tabular}{c | c |c | c }
	\hline
	        &             & RMSE                     & rank \\ \hline
	$c= 10$ & active      & 0.894$\pm$0.001          & 3    \\ \cline{2-4}
	        & ALT-Impute  & 0.900$\pm$0.006          & 3    \\ \cline{2-4}
	        & Soft-Impute & \textbf{0.893$\pm$0.001} & 3    \\ \cline{2-4}
	        & AIS-Impute  & \textbf{0.893$\pm$0.001} & 3    \\ \hline
	$c= 20$ & active      & \textbf{0.847$\pm$0.001} & 14   \\ \cline{2-4}
	        & ALT-Impute  & 0.850$\pm$0.001          & 14   \\ \cline{2-4}
	        & Soft-Impute & \textbf{0.847$\pm$0.001} & 14   \\ \cline{2-4}
	        & AIS-Impute  & \textbf{0.847$\pm$0.001} & 14   \\ \hline
	$c= 30$ & active      & \textbf{0.820$\pm$0.001} & 116  \\ \cline{2-4}
	        & ALT-Impute  & 0.825$\pm$0.001          & 116  \\ \cline{2-4}
	        & AIS-Impute  & \textbf{0.820$\pm$0.001} & 116  \\ \hline
\end{tabular}
\label{tab:netflix}
\vspace{-15px}
\end{table}

\subsection{Grayscale Images}

In this section,
we perform experiments on images from 
\cite{hu2013fast}
(Figures~\ref{fig:gray:rice}-\ref{fig:gray:wall}).
The
pixels are normalized
to zero mean and unit variance.
Gaussian noise from $\mathcal{N}(0, 0.05)$ 
is then added.
In each  image,
50\% of the pixels 
are randomly sampled as observations (half for training and another half for validation).
The task is to fill in the remaining 80\% of the pixels.
The experiment is repeated five times.

\begin{figure}[ht]
\centering
\vspace{5px}
\subfigure[\textit{rice} (854 $\times$ 960).
\label{fig:gray:rice}]
{\includegraphics[width = 0.30 \columnwidth]{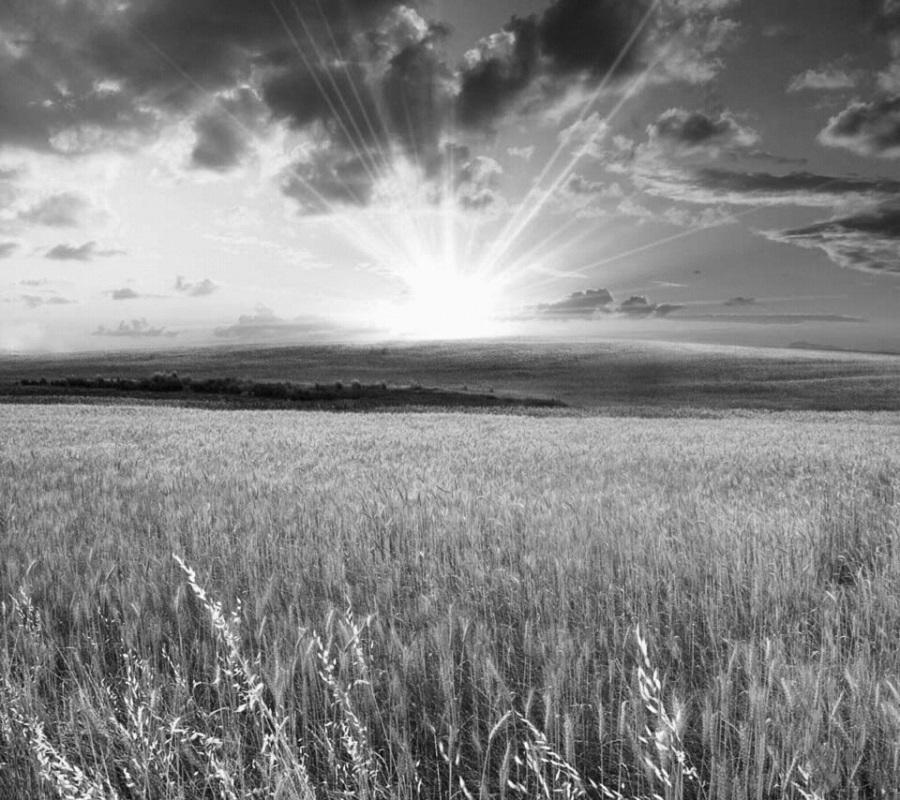}}
\subfigure[\textit{tree} (800 $\times$ 800).]
{\includegraphics[width = 0.30 \columnwidth]{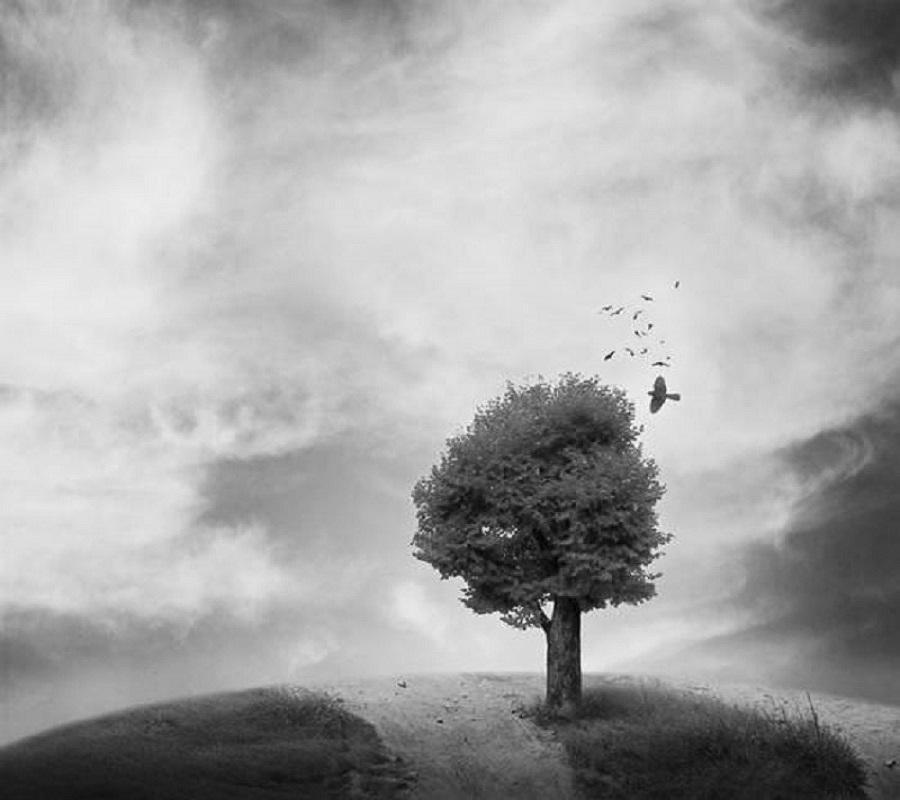}}
\subfigure[\textit{wall} (841 $\times$ 850).]
{\includegraphics[width = 0.30 \columnwidth]{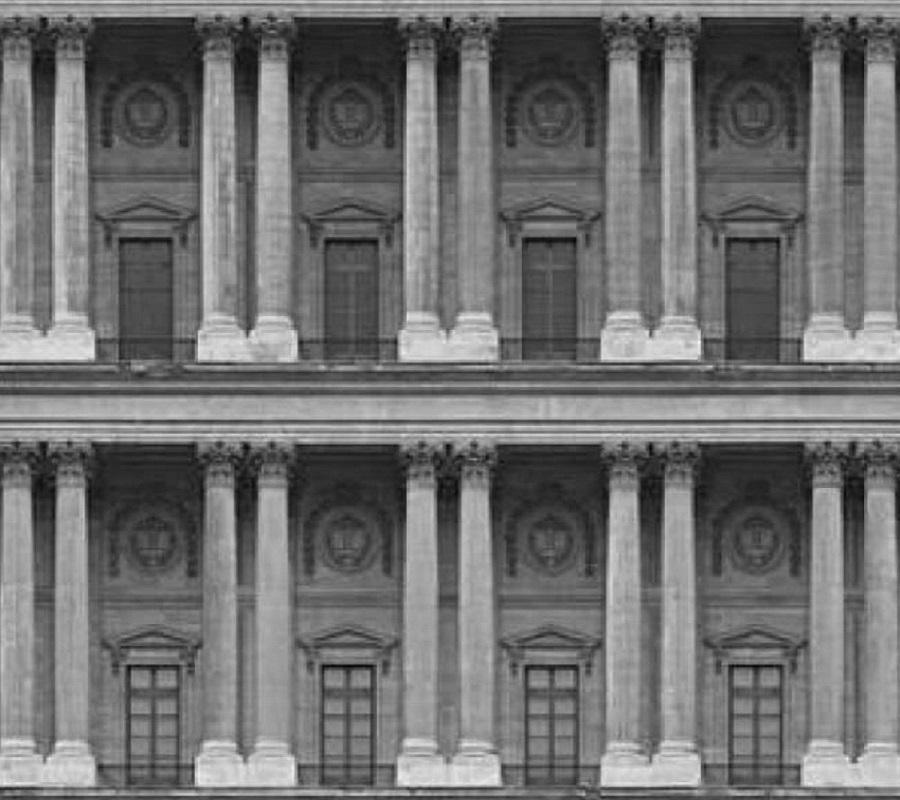}
	\label{fig:gray:wall}}

\vspace{-10px}
\caption{Grayscale images used for matrix completion.
	Their sizes are shown in the bracket.}
\vspace{-15px}
\end{figure}

Table~\ref{tab:gray} 
shows the testing RMSE and recovered rank.
As can be seen,
nuclear norm minimization
is better 
in terms of RMSE (in particular, AIS-Impute, ALT-Impute, APG and boost are the
best),
though
they require the use of higher ranks.
Figure~\ref{fig:greyimg} shows the
running time.
As can be seen, AIS-Impute is consistently the fastest.

Figure~\ref{fig:recvimg} compares the difference between recovered images from all algorithms 
and the clean one on image \textit{tree}.
As can be seen,
the difference on SSGD is the largest.
Besides, 
LMaFit, ASD, and R1MP and SSGD also have larger errors than the rest.
The observations on \textit{rice} and \textit{wall} are similar,
however,
due to space limitation,
we do not show them here.

\begin{figure*}[ht]
	\centering
	\vspace{-5px}
	\subfigure[\textit{rice}.]
	{\includegraphics[width= 0.24 \textwidth]{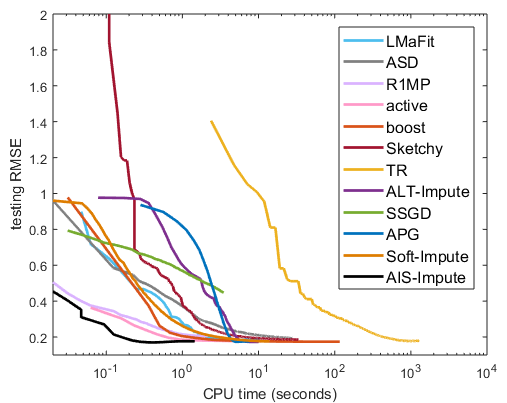}}
	\qquad
	\subfigure[\textit{tree}.]
	{\includegraphics[width= 0.24 \textwidth]{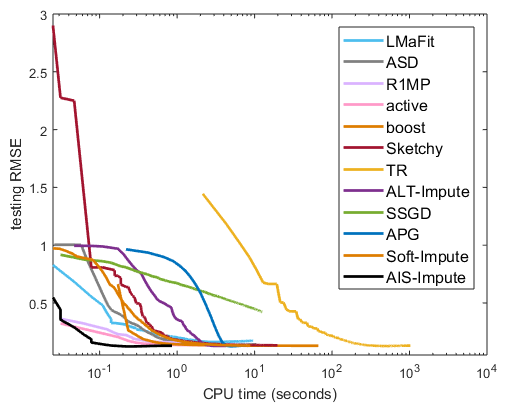}}
	\qquad
	\subfigure[\textit{wall}.]
	{\includegraphics[width= 0.24 \textwidth]{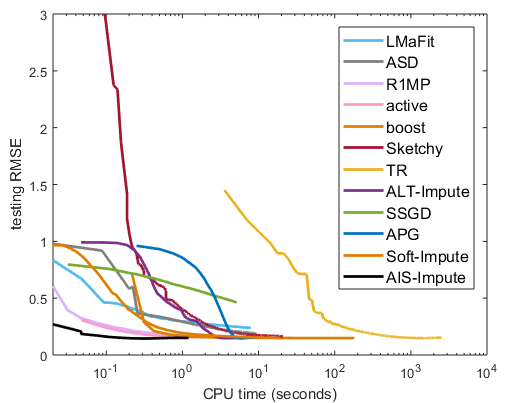}}	
	
	\vspace{-10px}
	\caption{Testing RMSE vs CPU time (in seconds) on grayscale images.}
	\label{fig:greyimg}
\end{figure*}

\begin{table*}[ht]
\centering
\vspace{-10px}
\caption{Matrix completion results on grayscale images. CPU time is in seconds.}
\vspace{-10px}
\begin{tabular}{c|c|c c|c c|c c}
	\hline
	   \multicolumn{2}{c|}{}    & \multicolumn{2}{c|}{\textit{rice}} & \multicolumn{2}{c|}{\textit{tree}} & \multicolumn{2}{c}{\textit{wall}} \\
	   \multicolumn{2}{c|}{}    &           RMSE           & rank    &           RMSE           & rank    &           RMSE           & rank   \\ \hline
	factorization &   LMaFit    &     0.189$\pm$0.002      & 45      &     0.174$\pm$0.013      & 25      &     0.238$\pm$0.004      & 50     \\ \cline{2-8}
	              &     ASD     &     0.194$\pm$0.020      & 45      &     0.142$\pm$0.004      & 25      &     0.189$\pm$0.012      & 50     \\ \cline{2-8}
	              &    R1MP     &     0.207$\pm$0.001      & 54      &     0.159$\pm$0.002      & 53      &     0.175$\pm$0.001      & 58     \\ \hline
	nuclear norm  &   active    & \textbf{0.176$\pm$0.002} & 100     & \textbf{0.130$\pm$0.002} & 71      & \textbf{0.150$\pm$0.002} & 101     \\ \cline{2-8}
	minimization  &    boost    & \textbf{0.176$\pm$0.004} & 94      & \textbf{0.130$\pm$0.002} & 60      & \textbf{0.149$\pm$0.002} & 93     \\ \cline{2-8}
	              &   Sketchy   & 0.186$\pm$0.007 & 89      &     0.134$\pm$0.002      & 41      &     0.157$\pm$0.008      & 88     \\ \cline{2-8}
	              &     TR      & \textbf{0.179$\pm$0.001} & 150     & \textbf{0.131$\pm$0.002} & 103     & \textbf{0.151$\pm$0.001} & 149     \\ \cline{2-8}
	              &    SSGD     &     0.447$\pm$0.058      & 96      &     0.424$\pm$0.037      & 60      &     0.463$\pm$0.023      & 96     \\ \cline{2-8}
	              &     APG     & \textbf{0.176$\pm$0.001} & 96      & \textbf{0.130$\pm$0.002} & 60      & \textbf{0.151$\pm$0.001} & 96     \\ \cline{2-8}
	              & Soft-Impute & \textbf{0.176$\pm$0.001} & 113     & \textbf{0.131$\pm$0.004} & 71      & \textbf{0.151$\pm$0.002} & 112     \\ \cline{2-8}
	              &  ALT-Impue  & \textbf{0.176$\pm$0.004} & 96      & \textbf{0.130$\pm$0.004} & 71      & \textbf{0.150$\pm$0.001} & 95     \\ \cline{2-8}
	              & AIS-Impute  & \textbf{0.176$\pm$0.001} & 96      & \textbf{0.219$\pm$0.002} & 70      & \textbf{0.150$\pm$0.001} & 95     \\ \hline
\end{tabular}
\label{tab:gray}
\end{table*}

\begin{figure*}[ht]
\centering
\vspace{-5px}
\includegraphics[width= 0.95 \textwidth]{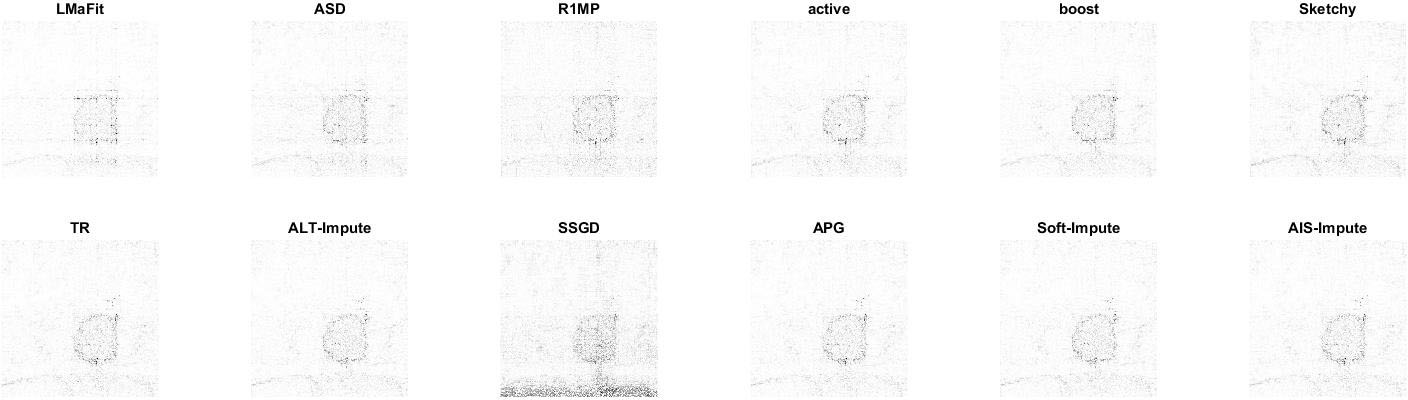}

\vspace{-10px}
\caption{Comparison on the difference between reconstructed images and the clean one on image \textit{tree}.}
\label{fig:recvimg}
\vspace{-5px}
\end{figure*}


\subsection{Nonconvex Regularization}

In the following,
we first perform experiments on (i) synthetic data,
using the setup in Section~5.1 (with $m = 250$ and $1000$); and
(ii) the recommender data set \textit{MovieLens-100K},
using the setup in Section~5.2.
Three nonconvex low-rank regularizers are considered,
namely, truncated nuclear norm (TNN) \cite{hu2013fast},
capped-$\ell_1$ norm \cite{zhang2010analysis} and log-sum-penalty (LSP) \cite{candes2008enhancing}.

For TNN,
we compare three solvers:
(i) TNNR(APG): the solver used in \cite{hu2013fast}; 
(ii) 
IRNN \cite{lu2016nonconvex}, which is a 
more recent proximal algorithm 
for optimization with nonconvex low-rank matrix regularizers (including the TNN); and
(iii) the proposed AIS-Impute extension (denoted DC(AIS-Impute)),
which replaces the original APG solver in \cite{hu2013fast} for the subproblem in TNNR with AIS-Impute.
For capped-$\ell_1$ and LSP,
two solvers are considered:
(i) IRNN and (ii) the proposed AIS-Impute extension.
As a further baseline, we also compare with 
(convex) nuclear norm regularization
with the AIS-Impute solver.
Experiments are repeated five times.

Results are shown in Table~\ref{tab:tnnr}.
As can be seen, the errors
obtained by nonconvex regularization
(i.e., TNN, capped-$\ell_1$ and LSP)
are much lower than those from convex nuclear norm regularization, illustrating the
advantage of using nonconvex regularization.
The performance obtained by the different nonconvex regularizers are comparable.

\begin{table*}[ht]
\centering
\vspace{-5px}
\caption{Comparison of 
	nuclear norm regularization
	with various nonconvex regularizations.}
\vspace{-10px}
\begin{tabular}{ c | c | c | c  | c  }
	\hline
	\multicolumn{2}{c|}{}        &                \multicolumn{2}{c|}{NMSE}                & RMSE                     \\
	\multicolumn{2}{c|}{}        & synthetic ($m = 250$)      & synthetic ($m = 1000$)     & \textit{MovieLens-100K}  \\ \hline
	nuclear norm    & AIS-Impute       & 0.0098$\pm$0.0004          & 0.0092$\pm$0.0002          & 0.883$\pm$0.005          \\ \hline
	TNN             & TNNR(APG)        & \textbf{0.0081$\pm$0.0004} & \textbf{0.0073$\pm$0.0001} & \textbf{0.851$\pm$0.002} \\ \cline{2-5}
	& IRNN             & \textbf{0.0081$\pm$0.0004} & \textbf{0.0073$\pm$0.0001} & 0.853$\pm$0.004          \\ \cline{2-5}
	& DC(AIS-Impute) & \textbf{0.0081$\pm$0.0004} & \textbf{0.0073$\pm$0.0002} & \textbf{0.851$\pm$0.002} \\ \hline
	capped-$\ell_1$ & IRNN             & 0.0089$\pm$0.0005          & \textbf{0.0074$\pm$0.0001} & 0.853$\pm$0.002          \\ \cline{2-5}
	& DC(AIS-Impute)   & \textbf{0.0081$\pm$0.0004} & \textbf{0.0073$\pm$0.0002} & 0.852$\pm$0.005          \\ \hline
	LSP             & IRNN             & 0.0083$\pm$0.0004          & 0.0076$\pm$0.0001          & 0.852$\pm$0.006          \\ \cline{2-5}
	& DC(AIS-Impute)   & \textbf{0.0081$\pm$0.0004} & \textbf{0.0073$\pm$0.0002} & \textbf{0.850$\pm$0.002} \\ \hline
\end{tabular}
\label{tab:tnnr}
\vspace{-15px}
\end{table*}

\begin{figure*}[ht]
\centering
\vspace{-5px}
\subfigure[TNN.]
{\includegraphics[width= 0.235 \textwidth]{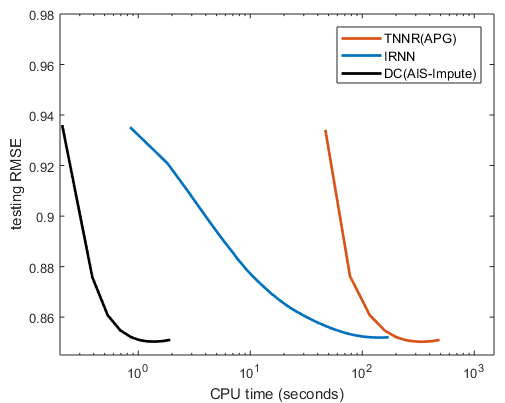}}
\qquad
\subfigure[capped-$\ell_1$.]
{\includegraphics[width= 0.24 \textwidth]{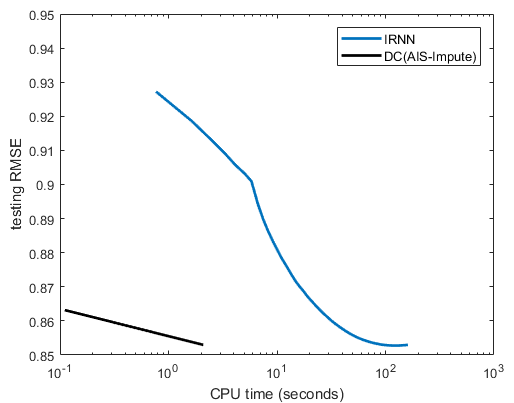}}
\qquad
\subfigure[LSP.]
{\includegraphics[width= 0.24 \textwidth]{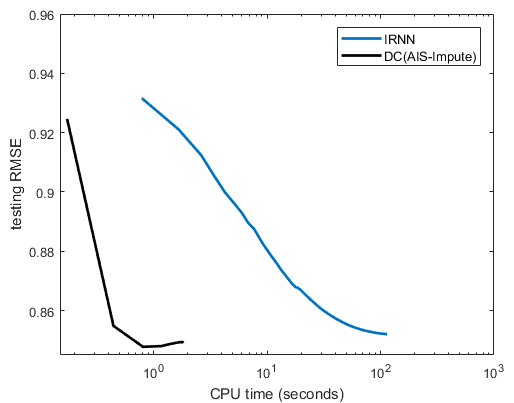}}	
\vspace{-10px}
\caption{Convergence of testing RMSE vs CPU time (in seconds) on the \textit{MovieLens-100K} data set.}
\label{fig:tnnr}
\vspace{-5px}
\end{figure*}

\subsection{Link Prediction}
\label{sec:explink}

Given a graph with $m$ nodes
and an incomplete
adjacency  matrix $O \in \{\pm 1\}^{m \times m}$,
link prediction aims to recover a low-rank matrix $X \in \R^{m \times m}$ 
such that
the signs of 
$X_{ij}$'s 
and $O_{ij}$'s
agree on most of the observed entries.
This is a binary matrix completion problem
\cite{chiang2014prediction}, and
we use the logistic loss $\ell(X_{ij}, O_{ij}) \equiv \log\left(1 + \exp(- X_{ij} O_{ij})\right)$ in \eqref{eq:mcgen}.

Experiments are performed on the
\textit{Epinions} 
and 
\textit{Slashdot} 
data sets
\cite{chiang2014prediction}
(Table~\ref{tab:linkprediction}).
Each row/column of the matrix $O$ corresponds to a user
(users with fewer than two observations are removed).
For \textit{Epinions}, $O_{ij} = 1$ if user $i$ trusts user $j$, and $-1$ otherwise.
Similarly 
for \textit{Slashdot}, $O_{ij} = 1$ if user $i$ tags user $j$ as friend, and $-1$ otherwise.
As can be seen from previous sections, Boost, TR, SSGD, APG and Soft-Impute are all slow, and thus
they are not considered here.
Besides, LMaFit and ALT-Impute are designed for the square loss.
Thus, comparison is performed with 
(i) active subspace selection; (ii) AIS-Impute; and (iii) AltMin: the alternative minimization approach used in \cite{chiang2014prediction}.
We use 80\% of the ratings for training, 
	10\% for validation and the rest for testing.
Let $X$ be the recovered matrix, 
and the test set $\{\hat{O}_{ij}\}$ be indexed by the set $\hat{\Omega}$. 
For performance evaluation, 
we use the 
(i) testing accuracy
$\frac{1}{\NM{\hat{\Omega}}{1}} \sum_{(i,j) \in \hat{\Omega}} I(\text{sign}( X_{ij} ) =
\hat{O}_{ij} )$,
where $I(\cdot)$ is the indicator function;
and
(ii) the rank of $X$.
To reduce statistical variability, 
experimental results are averaged over 5 repetitions.

\begin{table}[ht]
\centering
\vspace{-10px}
\caption{Data sets for link prediction.}
\vspace{-10px}
\begin{tabular}{c | c | c | c}
	\hline
	                  & \#rows & \#columns & \#signs \\ \hline
	\textit{Epinions} & 84,601 & 48,091    & 505,074 \\ \hline
	\textit{Slashdot} & 70,284 & 32,188    & 324,745 \\ \hline
\end{tabular}
\label{tab:linkprediction}
\vspace{-5px}
\end{table}

Results are shown in Table~\ref{tab:linkp:perf}
and 
Figure~\ref{fig:Epinions} shows the testing accuracy with CPU time. 
As can be seen,
active and AIS-Impute have slightly better accuracies than AltMin,
and AIS-Impute is the fastest.

\begin{table}[ht]
\centering
\vspace{-10px}
\caption{Performance on link prediction.}
\vspace{-10px}
\begin{tabular}{c | c|c | c}
	\hline
	                  &            &         accuracy         & rank \\ \hline
	\textit{Epinions} & active     & \textbf{0.939$\pm$0.002} & 12   \\ \cline{2-4}
	                  & AltMin     &     0.936$\pm$0.002      & 41   \\ \cline{2-4}
	                  & AIS-Impute & \textbf{0.940$\pm$0.001} & 12   \\ \hline
	\textit{Slashdot} & active     & \textbf{0.844$\pm$0.001} & 16   \\ \cline{2-4}
	                  & AltMin     &     0.839$\pm$0.002      & 39   \\ \cline{2-4}
	                  & AIS-Impute & \textbf{0.843$\pm$0.001} & 16   \\ \hline
\end{tabular}
\label{tab:linkp:perf}
\vspace{-20px}
\end{table}

\begin{table*}[ht]
\centering
\vspace{-5px}
\caption{Tensor completion results on the synthetic data.
	Here, sparsity is the proportion of observed entries, and
	post-processing
	time is in seconds.}
\vspace{-10px}
\begin{tabular}{c | c | c | c | c || c }
	\hline
               &                   &                \multicolumn{2}{c|}{NMSE}                &       & \\
               &                   & no post-processing               & with post-processing             & rank  & post-processing time      \\ \hline
  $m=125$      & APG               & \textbf{0.0162$\pm$0.0015} & \textbf{0.0100$\pm$0.0006} & 3,3,0 & 0.1       \\ \cline{2-6}
	(sparsity: 62.4\%) & Soft-Impute       & \textbf{0.0162$\pm$0.0014} & \textbf{0.0100$\pm$0.0005} & 3,3,0 & 0.1       \\ \cline{2-6}
               & AIS-Impute(exact) & \textbf{0.0161$\pm$0.0015} & \textbf{0.0100$\pm$0.0005} & 3,3,0 & 0.1       \\ \cline{2-6}
               & AIS-Impute        & \textbf{0.0159$\pm$0.0011} & \textbf{0.0099$\pm$0.0004} & 3,3,0 & 0.1       \\ \hline
  $m=500$      & APG               & \textbf{0.0166$\pm$0.0007} & \textbf{0.0105$\pm$0.0004} & 3,3,0 & 0.1       \\ \cline{2-6}
	(sparsity: 16.0\%) & Soft-Impute       & \textbf{0.0168$\pm$0.0007} & \textbf{0.0106$\pm$0.0004} & 3,3,0 & 0.1       \\ \cline{2-6}
               & AIS-Impute(exact) & \textbf{0.0167$\pm$0.0006} & \textbf{0.0104$\pm$0.0003} & 3,3,0 & 0.1       \\ \cline{2-6}
               & AIS-Impute        & \textbf{0.0167$\pm$0.0007} & \textbf{0.0105$\pm$0.0003} & 3,3,0 & 0.1       \\ \hline
  $m=2000$     & APG               & \textbf{0.0162$\pm$0.0013} & \textbf{0.0105$\pm$0.0006} & 3,3,0 & 0.5       \\ \cline{2-6}
	(sparsity: 3.9\%)  & Soft-Impute       & 0.0168$\pm$0.0016          & \textbf{0.0109$\pm$0.0011} & 3,3,0 & 0.4       \\ \cline{2-6}
               & AIS-Impute(exact) & \textbf{0.0161$\pm$0.0012} & \textbf{0.0104$\pm$0.0007} & 3,3,0 & 0.4       \\ \cline{2-6}
               & AIS-Impute        & \textbf{0.0161$\pm$0.0012} & \textbf{0.0104$\pm$0.0007} & 3,3,0 & 0.1       \\ \hline
\end{tabular}
\label{tab:tenperf}
\vspace{-10px}
\end{table*}

\begin{figure}[ht]
\centering
\subfigure[\textit{Epinions}.]
{\includegraphics[width=0.24\textwidth]{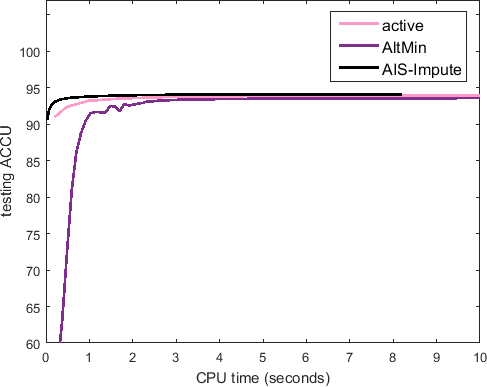}}
\subfigure[\textit{Slashdot}.]
{\includegraphics[width=0.235\textwidth]{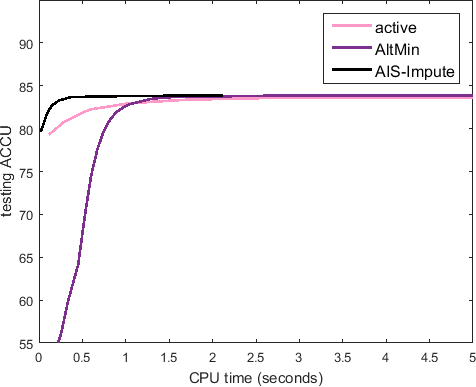}}

\vspace{-10px}
\caption{Testing accuracy vs CPU time (in seconds) on the \textit{Epinions} and \textit{Slashdot} data sets.}
\label{fig:Epinions}
\vspace{-10px}
\end{figure}

\subsection{Tensor Completion: Synthetic Data}
\label{sec:tensorsyn}

In this section, 
we perform tensor completion experiments with synthetic data.
The ground-truth data tensor 
(of size $m \times m \times 3$)
is generated as $\ten{O} = \ten{C} \times_1 A_1 \times_2 A_2 \times_3 A_3$,
where 
the elements of $A_1 \in \R^{m \times 3}$, $A_2 \in \R^{m \times 3}$, $A_3 \in \R^{3 \times 3}$
and
the core tensor $\ten{C} \in \R^{3 \times 3}$
are all sampled i.i.d. from the standard normal distribution $\mathcal{N}(0, 1)$,
and $\times_k$ is the $k$-mode product\footnote{The $k$-mode product of a tensor $\ten{X}$ and a matrix 
	$A$ is defined as $\ten{X} \times_k A = \left( \ten{X}_{\ip{k}} A \right)_{\ip{k}}$ \cite{kolda2009tensor}.}.
Thus,
$\ten{O}$ is low-rank for the first two mode but not for the third.
Noise 
$\ten{G}$,
with elements sampled i.i.d. from the normal distribution $\mathcal{N}(0, 0.05)$, is then added.
A total number of $\Omega = 45 m \log(m)$ random elements in $\ten{O}$ are observed.
	Half of them are used for training, and the other half for validation.
On testing, we perform evaluation on the unobserved entries and use the same criteria as in
Section~\ref{sec:matcomp:syn}, i.e., NMSE and recovered rank in each mode.

Similar to Section~\ref{sec:matcomp:syn}, 
we compare the following algorithms:
(i) APG;
(ii) extension of Soft-Impute to tensor completion,
which is based on Section~\ref{sec:extprox}; 
(iii) the proposed algorithm with exact SVD (AIS-Impute(exact));
and (iv) the proposed algorithm which uses power method to approximate SVT (AIS-Impute).

As suggested in \cite{wimalawarne2014multitask},
we set $(\lambda_1, \lambda_2, \lambda_3)$ in
the scaled latent nuclear norm to
$(1, 1, \frac{\sqrt{m}}{\sqrt{3}}) \lambda$.
Thus, the only tunable parameter is $\lambda$,
which is obtained by grid search using the validation set.
We also vary $m$ in $\{ 125, 500, 2000 \}$.
Experimental results are averaged over 5 repetitions.

Results on NMSE and rank are shown in Table~\ref{tab:tenperf}.
As can be seen,
APG, Soft-Impute, AIS-Impute(exact) and AIS-Impute have comparable performance.
The plots of objective value vs time
and iterations are shown in Figure~\ref{fig:tenconv}.
In terms of iterations,
APG, AIS-Impute(exact) and AIS-Impute have similar behavior
as they all have $O(1/T^2)$ convergence rate.
These also agree with the matrix case in Section~\ref{sec:matcomp:syn}.
In terms of time,
as APG does not utilize the ``sparse plus low-rank'' structure,
it is slower than AIS-Impute(exact) and AIS-Impute. 
AIS-Impute is the fastest,
as it has both fast $O(1/T^2)$ convergence rate and low per-iteration complexity.

\begin{figure}[ht]
	\centering
	\vspace{-5px}
	\subfigure[$m = 125$.]
	{\includegraphics[width = 0.2425\textwidth]{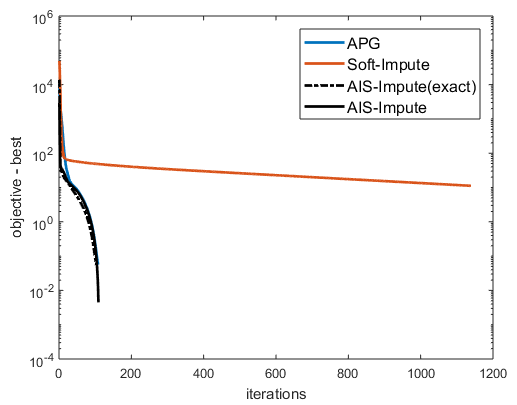}
		\includegraphics[width = 0.24\textwidth]{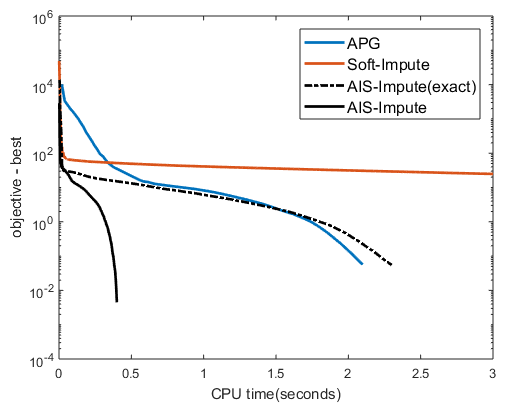}}
	\subfigure[$m = 500$.]
	{\includegraphics[width = 0.2425\textwidth]{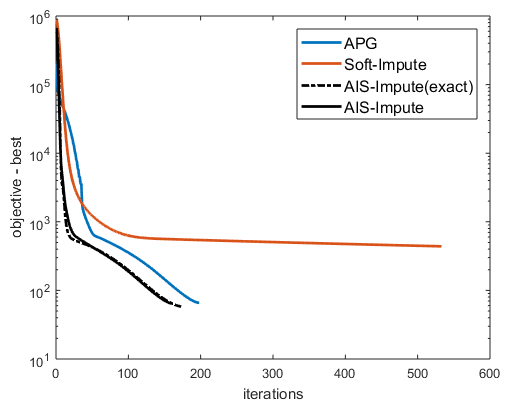}
		\includegraphics[width = 0.24\textwidth]{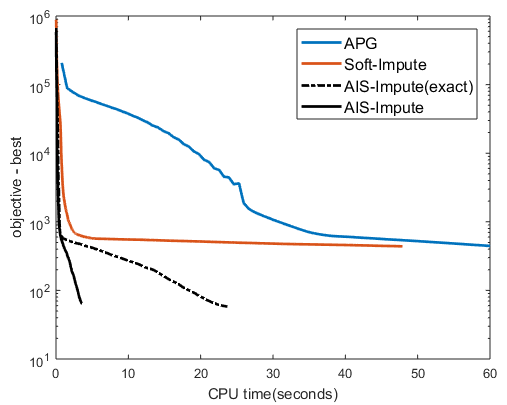}}
	\subfigure[$m = 2000$.]
	{\includegraphics[width = 0.24\textwidth]{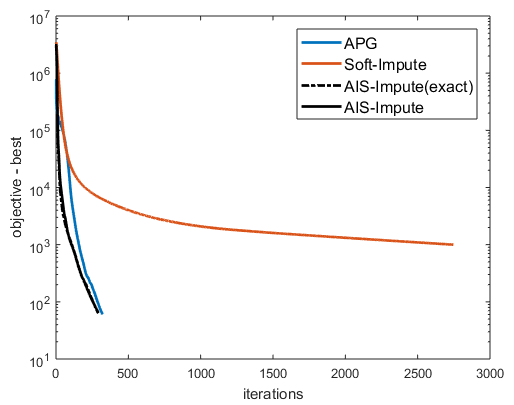}
		\includegraphics[width = 0.24\textwidth]{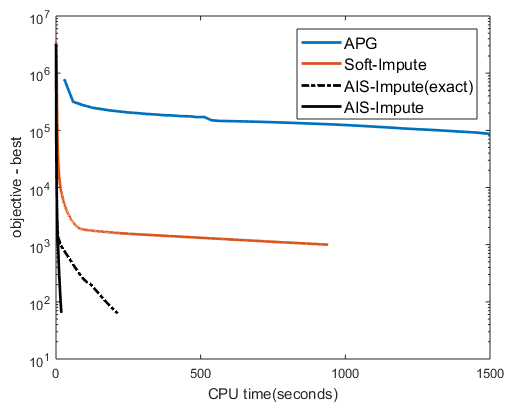}}
	
	\vspace{-10px}
	\caption{Convergence of objective value on the synthetic tensor data.
		Left: vs number of iterations;
		Right: vs CPU time (in seconds).}
	\label{fig:tenconv}
	\vspace{-10px}
\end{figure}

Performance with post-processing in Section~\ref{sec:postproc} is shown in
Table~\ref{tab:tenperf}.
As can be seen,
it is very efficient and improves NMSE.
Thus, 
we always perform post-processing in the sequel.

\subsection{Multi-Relational Link Prediction}
\label{sec:multen}

In this section, we
perform experiments on the \textit{YouTube} 
data set
\cite{lei2009analysis}. 
It contains 15,088 users, and describes five types of user interactions:
contact,
number of shared friends, 
number of shared subscriptions, 
number of shared subscribers, 
and the number of shared favorite videos.
Thus, it forms a 
$15088 \times 15088 \times 5$
tensor,
with a total of $27,257,790$ nonzero elements.
Following \cite{chiang2014prediction}, 
we 
formulate multi-relational link prediction as a tensor completion problem.
As the observations are real-valued,
we use the square loss in \eqref{eq:tc}.
Besides AIS-Impute (Algorithm~\ref{alg:AISImpute:ten}), 
we also compare with the following state-of-the-art non-proximal-based tensor completion algorithms:
(i)
geometric nonlinear CG for tensor completion (denoted ``GeomCG'') \cite{kressner2014low}:
a gradient descent approach with gradients restricted on the Riemannian manifold;
(ii)
An alternating direction method of multipliers approach (denoted ``ADMM(overlap)") \cite{tomioka2010estimation},
which solves the overlapping nuclear norm regularized tensor completion problem;
(iii) 
 fast low rank tensor completion (denoted ``FaLRTC'') \cite{liu2013tensor}:
It smooths the overlapping nuclear norm 
and then solves the relaxed problem with accelerated gradient descent; and
(iv)
tensor completion by parallel matrix factorization
(denoted ``TMac'') \cite{Xu2013}:
An extension of LMaFit \cite{wen2012solving} to tensor completion,
which performs simultaneous low-rank matrix factorizations to all mode matricizations.

\noindent
\textbf{YouTube Subset.}
First,
we perform experiments on a small \textit{YouTube} subset,
obtained by random selecting 1000 users 
(leading to $12,101$ observations).
We use $50 \%$ of the observations for training, another $25 \%$ for validation and the remaining for testing.
Let $\ten{X}$ be the recovered tensor,
and the testing ratings $\hat{O}_{ij}$ be indexed by the set $\hat{\Omega}$.
For performance evaluation, 
we use (i) the testing root mean squared error 
$\text{RMSE} = \sqrt{\NM{P_{\hat{\Omega}}(\ten{X} - \hat{\ten{O}})}{F}^2 / \NM{\hat{\Omega}}{1}}$; 
and (ii) rank of the unfolded matrix in each mode. 
The experiments are repeated five times.

Performance is shown in Table~\ref{tab:youtubesml}
and 
Figure~\ref{fig:youtubesml}
shows the time comparison.
ADMM(overlap) and FaLRTC have similar recovery performance, but are all very slow due to usage of the SVD.
As the overlapping nuclear norm is smoothed in FaLRTC,
its cannot exactly recover a low-rank tensor.
TMac is fast,
but has the worst recovery performance.
AIS-Impute enjoys fast speed and good recovery performance.

\begin{table}[ht]
\centering
\vspace{-10px}
\caption{Results on the \textit{YouTube} subset.
	The rank is for each mode.}
\label{tab:youtubesml}
\vspace{-10px}
\begin{tabular}{ c | c | c }
	\hline
	              & RMSE                     & rank          \\ \hline
	GeomCG        & 0.672$\pm$0.050          & 7, 7, 5       \\ \hline
	ADMM(overlap) & 0.690$\pm$0.030          & 142, 142, 5   \\ \hline
	FaLRTC        & 0.672$\pm$0.032          & 1000, 1000, 5 \\ \hline
	TMac          & 0.786$\pm$0.027          & 4, 4, 0       \\ \hline
	AIS-Impute    & \textbf{0.616$\pm$0.029} & 33, 33, 0     \\ \hline
\end{tabular}
\vspace{-5px}
\end{table}


\noindent
\textbf{Full YouTube Data.}
Next, we perform experiments on the full \textit{YouTube} data set
with the same setup.
As ADMM(overlap) and FaLRTC are too slow,
we only compare with GeomCG, TMac and AIS-Impute.
Experiments are repeated five times.

Results are shown in Table~\ref{tab:youtubelrg},
and
Figure~\ref{fig:youtubeful}
shows the time.
TMac has much worse performance than GeomCG and AIS-Impute.
GeomCG is based on the 
(nonconvex)
Turker decomposition, and
its convergence rate is unknown.
Moreover, 
its iteration time complexity has a worse dependency on the tensor rank than AIS-Impute 
($\prod_{i = 1}^{D}r_t^d$ vs $\sum_{i = 1}^D r_t^d$), and
thus GeomCG becomes very slow when the tensor rank is large.
Overall,
AIS-Impute has fast speed and good recovery performance.

\begin{table}[ht]
\centering
\vspace{-10px}
\caption{Results on the full \textit{YouTube} dataset.
The rank is for each mode.}
\label{tab:youtubelrg}
\vspace{-10px}
\begin{tabular}{ c | c | c }
	\hline
	           & RMSE                     & rank    \\ \hline
	GeomCG     & 0.388$\pm$0.001          & 51, 51, 5 \\ \hline
	TMac       & 0.611$\pm$0.007          & 10, 10, 0 \\ \hline
	AIS-Impute & \textbf{0.369$\pm$0.006} & 70, 70, 0 \\ \hline
\end{tabular}
\vspace{-5px}
\end{table}

\begin{figure}[ht]
\centering
\subfigure[data subset.]
{\includegraphics[width = 0.2375\textwidth]{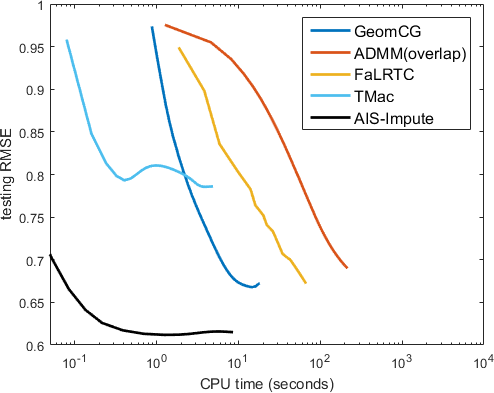}
	\label{fig:youtubesml}}
\subfigure[full data set.]
{\includegraphics[width = 0.235\textwidth]{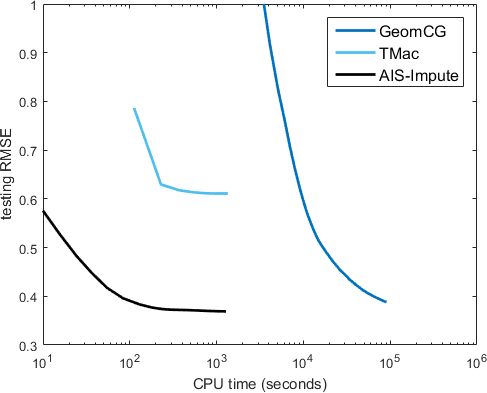}
	\label{fig:youtubeful}}

\vspace{-10px}
\caption{Testing RMSE vs CPU time on the \textit{Youtube} data set.}
\label{fig:youtube}
\vspace{-10px}
\end{figure}


\section{Conclusion}

In this paper, we show that 
Soft-Impute, as a proximal algorithm,
can be accelerated  without losing 
the ``sparse plus low-rank'' structure crucial to its efficiency.
To further reduce the per-iteration time complexity,
we proposed an approximate-SVT scheme based on the power method.
Theoretical analysis shows that the proposed algorithm still enjoys the fast $O(1/T^2)$
convergence rate.
We also extend the proposed algorithm for low-rank tensor completion 
with the scaled latent nuclear norm regularizer.
Again, the ``sparse plus low-rank'' structure 
can be preserved
and 
a convergence rate of
$O(1/T^2)$ can be obtained.
The proposed algorithm can be further extended to 
nonconvex low-rank regularizers,
which have better empirical performance than the convex nuclear norm regularizer.
Extensive experiments on both synthetic and real-world data sets 
show that the proposed algorithm is much faster than the state-of-the-art.

\section*{Acknowledgment}

This research was supported in part by
the Research Grants Council of the Hong Kong Special Administrative Region
(Grant 614513).


{
\bibliographystyle{IEEEtran}
\bibliography{bib}
}

\vspace{-30px}

\begin{IEEEbiography}[{\includegraphics[width = 1\textwidth]{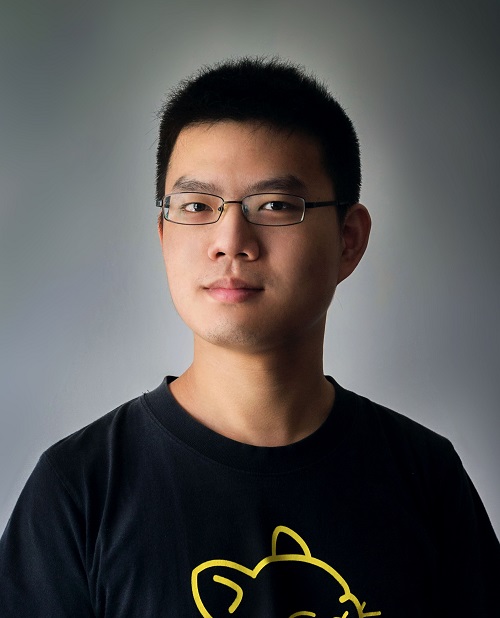}}]{Quanming Yao}
obtained his Ph.D from Computer Science 
and Engineer Department of Hong Kong University of Science and Technology (HKUST) in 2018,
and bachelor degree in Electronic and Information Engineering from the Huazhong University of Science and Technology (HUST) in
2013. His research interests focus on machine learning.
Currently, he is a research scientist in 4Paradigm Inc. (Beijing, China). 
He was awarded as Qiming star of HUST in 2012,
Tse Cheuk Ng Tai research excellence prize from HKUST in 2015
and Google PhD fellowship (machine learning) in 2016.
\end{IEEEbiography}

\vspace{-30px}

\begin{IEEEbiography}[{\includegraphics[width = 1\textwidth]{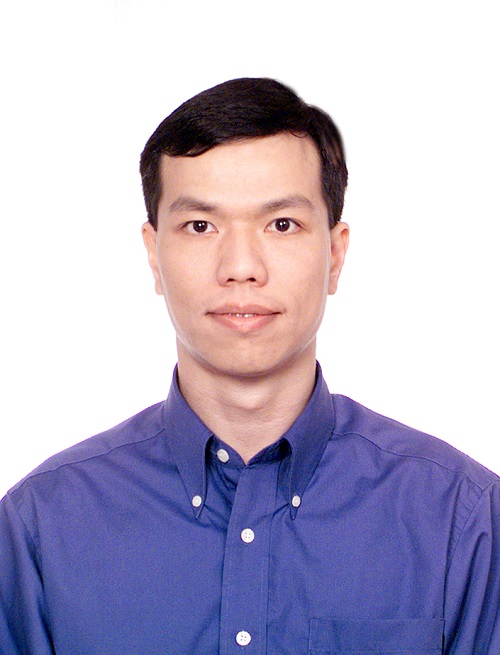}}]{James T. Kwok}
received the PhD degree in
computer science from the Hong Kong University
of Science and Technology in 1996. He was
with the Department of Computer Science,
Hong Kong Baptist University, Hong Kong, as
an assistant professor. He is currently a professor
in the Department of Computer Science and
Engineering, Hong Kong University of Science
and Technology. His research interests include
kernel methods, machine learning, example recognition,
and artificial neural networks. He
received the IEEE Outstanding 2004 Paper Award, and the Second
Class Award in Natural Sciences by the Ministry of Education, People’s
Republic of China, in 2008. He has been a program cochair for a number
of international conferences, and served as an associate editor for
the IEEE Transactions on Neural Networks and Learning Systems
and Neurocomputing journal.
\end{IEEEbiography}


\cleardoublepage
\appendices

\section{Proofs}
\label{app:proof}


\subsection{Proposition~\ref{pr:lipmac}}
\label{app:lipmac}

\begin{proof}
For any $X, Y \in \R^{m \times n}$,
\begin{align}
& \NM{\nabla f(X) - \nabla f(Y)}{F}^2
\notag 
\\
& = \sum_{(i,j) \in \Omega}\left[ \frac{d \ell(X_{ij}, O_{ij})}{d X_{ij}} - \frac{d \ell(Y_{ij}, O_{ij})}{d Y_{ij}} \right]^2
\notag
\\
& \le \sum_{(i,j)\in\Omega} \rho^2 (X_{ij} - Y_{ij})^2
\label{eq:temp16}
\\
& \le \rho^2 \NM{X - Y}{F}^2,
\notag
\end{align}
where \eqref{eq:temp16} follows from the fact that $\ell$ is $\rho$-Lipschitz smooth.
Thus, $f(X)$ is $\rho$-Lipschitz smooth.
\end{proof}

%
%
%


\subsection{Proposition~\ref{pr:approGSVT}}
\label{app:approGSVT}

\begin{proof}
	First,
	we introduce the following theorem.
	
	\begin{theorem}
		[Separation theorem	\cite{golub2012matrix}]
		\label{thm:sepsv}
		Let $A \in \R^{m \times n}$ and $B \in \R^{m \times r}$ with $B^{\top} B = I$.  Then
		\begin{align*}
		\sigma_{i}\left(  B^{\top} A \right) 
		\le \sigma_i(A),
		\;\text{for}\; 
		i = 1, \dots, \min(r, n).
		\end{align*}
	\end{theorem}
	Let the SVD of $Z$ be $U \Sigma V^{\top}$. 
	$Z$ can then be rewritten as
	\begin{equation} \label{eq:tmp2}
	Z = 
	\left[ U_{\breve{k}_t}; U_{\bot} \right] 
	\begin{bmatrix}
	\Sigma_{\breve{k}_t} & \\
	& \Sigma_{\bot}
	\end{bmatrix}
	\left[ V_{\breve{k}_t}; V_{\bot} \right]^{\top},
	\end{equation} 
	where $U_{\breve{k}_t}$ contains the $\breve{k}_t$ leading columns of $U$, 
	and $U_{\bot}$ the remaining columns.
	Similarly,
	$\Sigma_{\breve{k}_t}$ (resp.
	$V_{\breve{k}_t}$) contains the $\breve{k}_t$ leading eigenvalues (resp. columns) of 
	$\Sigma$ (resp.
	$V$).
	Then, let
	\begin{align}
	\tilde{u}_i =  Q^{\top} u_i 
	\text{\quad and \quad}
	\tilde{v}_i = v_i,
	\label{eq:temp111}
	\end{align}
	where $u_i$ (resp. $v_i$) is the $i$th column of $U$ (resp. $V$).
	For $i = 1, \cdots, \breve{k}_t$, we have
	\begin{eqnarray}
	\tilde{u}_i^{\top} \left(  Q^{\top} Z \right)  \tilde{v}_i
	& = &
	U_i^{\top} \left(  Q Q^{\top} \right) 
	Z V_i
	\notag
	\\
	& = &
	U_i^{\top} Z V_i
	\label{eq:temp32}
	\\
	& =  &
	\sigma_i(Z),
	\label{eq:temp888}
	\end{eqnarray}
	where \eqref{eq:temp32} is due to $\Span{ U_{\breve{k}_t} } \subseteq \Span{Q}$.
	Hence,
	\begin{eqnarray}
	\sigma_i\left( Q^{\top} Z \right)  
	& = & \sigma_i(Z),
	\text{\;for\;}
	i = 1, \cdots, \breve{k}_t
	\label{eq:temp28}
	\end{eqnarray}
	From Theorem~\ref{thm:sepsv},
	by substituting $Q = B$ and $A = Z$, we have
	$\sigma_i( Q^{\top} Z )  \le \sigma_i(Z)$.
	Combining with \eqref{eq:temp28},
	we obtain that 
	the rank-$\breve{k}_t$ SVD of $Q^{\top} Z$ is
	$( Q^{\top} U_{\breve{k}_t} )  \Sigma_{\breve{k}_t} V_{\breve{k}_t}^{\top}$,
	with the corresponding left and right singular vectors contained in $Q^{\top}
	U_{\breve{k}_t}$ and
	$V_{\breve{k}_t}$, respectively.

	Again,
	by Theorem~\ref{thm:sepsv},
	we have 
	\begin{align*}
	\sigma_{\breve{k}_t + 1}\left( Q^{\top} Z \right) 
	\le 
	\sigma_{\breve{k}_t + 1}(Z)
	\le \mu.
	\end{align*}
	Besides, using (\ref{eq:tmp2}),
	\begin{align*}
	\sigma_i\left( Q^{\top} Z \right) 
	=  \max_{ u , v } 
	u \left(  Q^{\top} U_{\breve{k}_t} \Sigma_{\breve{k}_t}  V_{\breve{k}_t}^{\top} 
	+ Q^{\top} U_{\bot} \Sigma_{\bot}  V_{\bot}^{\top} \right)  
	v.
	\notag
	\end{align*}
	Since the first $\breve{k}_t$ singular values are from the term
	$Q^{\top} U_{\breve{k}_t} \Sigma_{\breve{k}_t} V_{\breve{k}_t}$,
	then
	\begin{align}
	\sigma_{\breve{k}_t + 1}\left( Q^{\top} Z \right) 
	= \max_{u, v}
	u^{\top} 
	\left(  Q^{\top} U_{\bot} \Sigma_{\bot}  V_{\bot}^{\top} \right) 
	v
	\le \mu.
	\label{eq:temp36}
	\end{align}
	Then,
	\begin{align}
	& \svt_{\mu r} (Q^{\top} Z)
	\notag
	\\
	& = \svt_{\mu r}
	\left( 
	Q^{\top} U_{\breve{k}_t} \Sigma_{\breve{k}_t} V_{\breve{k}_t}^{\top}
	+ Q^{\top} U_{\bot} \Sigma_{\bot}  V_{\bot}^{\top}
	\right) 
	\notag
	\\
	& = \svt_{\mu r}
	\left( Q^{\top} U_{\breve{k}_t} \Sigma_{\breve{k}_t}  V_{\breve{k}_t}^{\top} \right) 
	+
	\svt_{\mu r}
	\left( Q^{\top} U_{\bot} \Sigma_{\bot}  V_{\bot}^{\top} \right) 
	\label{eq:temp37}
	\\
	& = \svt_{\mu r}
	\left( Q^{\top} U_{\breve{k}_t} \Sigma_{\breve{k}_t}  V_{\breve{k}_t}^{\top} \right) .
	\label{eq:temp34}
	\end{align}
	where \eqref{eq:temp37} follows from that $Q^{\top} U_{\breve{k}_t}$ (resp. $V_{\breve{k}_t}$) is orthogonal to $Q U_{\bot}$ (resp. $V_{\bot}$).
	\eqref{eq:temp36} shows that there are only $\breve{k}_t$ singular values in $Q^{\top} Z$ larger than $\mu$.
	Thus, $\svt_{\mu r}{ Q^{\top} U_{\bot} \Sigma_{\bot}  V_{\bot}^{\top} } = 0$
	and we get \eqref{eq:temp34}.
	Finally,
	\begin{align}
	Q \svt_{\mu r}
	\left( Q^{\top} Z \right) 
	& = Q \left(  Q^{\top} U_{\breve{k}_t} \svt_{\mu r}{\Sigma_{\breve{k}_t}} V_{\breve{k}_t}^{\top} \right)  
	\notag
	\\
	& = U_{\breve{k}_t} \svt_{\mu r} 
	\left( \Sigma_{\breve{k}_t} \right) 
	V_{\breve{k}_t}^{\top}
	\label{eq:temp35}
	\\
	&
	= \svt_{\mu r} (Z),
	\label{eq:temp33}
	\end{align}
	where \eqref{eq:temp35} comes from $\Span{U_{\breve{k}_t}} \subseteq \Span{Q}$;
	\eqref{eq:temp33}
	comes from that rank-$\breve{k}_t$ SVD of $Z$ is 
	$U_{\breve{k}_t} \Sigma_{\breve{k}_t} V_{\breve{k}_t}^{\top}$
	and $Z$ only has $\breve{k}_t$ singular values larger than $\mu$.
\end{proof}

\subsection{Proposition~\ref{pr:inexact}}
\label{app:inexact}

Before proof of Proposition~\ref{pr:inexact},
we first 
introduce some Lemmas
(Lemma~\ref{lem:expaness}, \ref{lem:temp1}, \ref{lem:powermethod} and \ref{pr:temp2})
and Propositions
(Proposition~\ref{pr:subbound} and \ref{pr:proxrate}).

\begin{lemma}[\cite{cai2010singular}] \label{lem:expaness}
For any matrices $A$ and $B$,
	$\FN{\svt_\l(A) - \svt_\l(B)} \le \FN{A - B}$.
\end{lemma}

Let $Z_t^* \equiv \svt_{\mu \l}(\tZ_t)$, $\beta_t \equiv \NM{\tZ_t}{F}$
and $\eta_t = \frac{\sigma_{k + 1}(\tZ_t)}{\sigma_k (\tZ_t)}$.

\begin{lemma}[\cite{halko2011finding}] \label{lem:powermethod} 
	Let the input to Algorithm~\ref{alg:powermethod} be $\tZ_t$,
	and its top $k$ left singular vectors be contained in $U_k$.
	Then, for $j = 0, 1, 2, \dots$,
	\[ \NM{Q_{j} {Q_{j}}^{\top} - U_k U_k^{\top}}{F} \le \eta_t^j \alpha_t,\]
	where $\alpha_t = \NM{Q_0 {Q_0}^{\top} - U_k U_k^{\top}}{F}$
	and $Q_0$ is the span of $\tZ_t R_t$.
\end{lemma}

\begin{lemma} \label{lem:temp1}
For output $\tilde{X} = (QU) \Sigma V^{\top}$ from Algorithm~\ref{alg:apprSVT},
we have $\NM{\tilde{X} - Z_t^*}{F} \le \NM{U_k U_k^{\top} - Q Q^{\top}}{F} \beta_t$.
\end{lemma}

\begin{proof}
From Proposition~\ref{pr:approGSVT}, 
\begin{eqnarray*}
	 Z_t^* 
	-\tilde{X} 
	& = & \svt_{\mu\l}(\tZ_t) - Q \svt_{\mu\l}(Q^{\top}\tZ_t)
	\notag \\
	& = & \svt_{\mu\l}(U_k U_k^{\top} \tZ_t) - \svt_{\mu\l}(Q Q^{\top}\tZ_t).
	\label{eq:temp3}
\end{eqnarray*}
Using Lemma~\ref{lem:expaness} and the Cauchy's inequality, 
\begin{eqnarray*} 
\NM{\tilde{X} - Z_*}{F}
& = & \NM{\svt_{\mu\l}(U_k U_k^{\top} \tZ_t) - \svt_{\mu\l}(Q Q^{\top}\tZ_t)}{F}
 \notag \\
& \le & \NM{(U_k U_k^{\top} - Q Q^{\top})\tZ_t}{F}
\notag \\
&  \le & \NM{U_k U_k^{\top} - Q Q^{\top}}{F} \beta_t,
\end{eqnarray*}
and result follows.
\end{proof}

\begin{proposition} 
\label{pr:subbound}
Let $G_t \in \partial h_{\mu\lambda \|\cdot\|_*}(\tilde{X}; \tZ_t)$,
then $\NM{G_t}{F}$
is upper-bounded by a constant $\gamma_t$.
\end{proposition}

\begin{proof}
	Let the reduced SVD of $\tilde{X}$ be $U \Sigma V^{\top}$
	(only positive singular values are contained).
	By the definition of subgradient of the nuclear norm \cite{candes2009exact},
	\[ \partial h_{\mu \lambda \|\cdot\|_*}(\tilde{X}; \tZ_t) = \tilde{X} - \tZ_t + \mu \lambda (U
	V^{\top} + W), \]
	where 
	\begin{equation} \label{eq:W}
	W^{\top} U = 0, \; W V= 0, \text{ and } \NM{W}{\infty} \le 1. 
	\end{equation} 
	
	Thus,
	\begin{align}
	\NM{G_t}{F}
	&   =  \NM{\tilde{X} - \tZ_t + \mu \lambda (U V^{\top} + W)}{F} 
	\notag \\
	& \le \NM{\tilde{X} - \tZ_t}{F} + \mu \lambda \NM{U V^{\top} + W}{F}. \label{eq:temp5}
	\end{align}

	For the first term in \eqref{eq:temp5},
	\begin{eqnarray}
	\lefteqn{\NM{\tilde{X} - \tZ_t}{F}}
	\notag \\
	&  = & \NM{\tilde{X} - Z_t^* + Z_t^* - \tZ_t}{F}
	\notag \\
	&\le & \NM{\tilde{X} - Z_t^*}{F} + \NM{Z_t^* - \tZ_t}{F}
	\notag \\
	& = & \NM{Z_t^* - \tZ_t}{F} + \NM{Z_t^* -  Q \svt_{\mu\l}(Q^{\top} \tZ_t) }{F}
	\notag \\
	& = & \NM{Z_t^* - \tZ_t}{F} + \NM{Z_t^* -  \tilde{X} }{F}
	\notag \\
	& \le & \NM{Z_t^* - \tZ_t}{F} + \NM{U_k U_k^{\top} - Q Q^{\top}}{F} \beta_t
	\label{eq:temp6}\\
	& \le & \NM{Z_t^* - \tZ_t}{F} + \alpha_t \beta_t.
	\label{eq:temp8}
	\end{eqnarray}
Here, \eqref{eq:temp6} follows from Lemma~\ref{lem:temp1}, and \eqref{eq:temp8} from Lemma~\ref{lem:powermethod}.
	As $\NM{W}{\infty} \le 1$ from \eqref{eq:W},
	thus
	\begin{align*}
	\NM{W}{F} = \sqrt{\sum_{i = 1}^m \sigma_i^2(W)} \le \sqrt{m}.
	\end{align*}
	
	For the second term in (\ref{eq:temp5}), then
	\begin{align} 
	\NM{U V^{\top} + W}{F}
	& \le \sqrt{\tr{U^{\top} U V^{\top} V}} + \NM{W}{F}
	\notag \\
	& \le \sqrt{k_t} + \sqrt{m} \le 2\sqrt{m}.
	\label{eq:temp7}
	\end{align}
	
	Combining (\ref{eq:temp8}) and (\ref{eq:temp7}), by Lemma~\ref{lem:powermethod}:
	\begin{align}
	\NM{G_t}{F} 
	\le 2 \mu \lambda \sqrt{m} + \NM{Z_t^* - \tZ_t}{F} + \alpha_t \beta_t.
	\label{eq:gammabnd}
	\end{align}
	Since $Z_t^*$ is independent of $\tilde{X}$,
	$\NM{Z_t^* - \tZ_t}{F}$ is a constant. Hence,
	$\NM{G_t}{F}$ is upper bounded by
	\begin{align*}
	\gamma_t = 2 \mu \lambda \sqrt{m} + \NM{Z_t^* - \tZ_t}{F} + \alpha_t \beta_t,
	\end{align*}
	which a constant.
\end{proof}

\begin{proposition} 
\label{pr:proxrate}
Assume that $k_t \ge \breve{k}_t$. 
Let $h_{\mu\lambda \|\cdot\|_*}(\tilde{X} ; \tZ_t)$ be as defined in \eqref{eq:h}.
Then,
for Algorithm~\ref{alg:apprSVT}, we have
\begin{align*} 
h_{\mu\lambda \|\cdot\|_*}(\tilde{X} ; \tZ_t) 
\le h_{\mu\lambda \|\cdot\|_*}(Z_t^*;\tZ_t)
+ \alpha_t \beta_t \gamma_t \eta_t^J.
\end{align*} 
\end{proposition}

\begin{proof}
As $h$ is convex, 
\begin{equation} \label{eq:temp1}
h_{\mu\lambda \|\cdot\|_*} (\tilde{X}; \tZ_t)
\le h_{\mu\lambda \|\cdot\|_*}(Z_t^*; \tZ_t) + \tr{(\tilde{X} - Z_t^*)^\top G_t} 
\end{equation}
where $G_t \in \partial h_{\mu\lambda \|\cdot\|_*}(\tilde{X}; \tZ_t)$.
Next, we bound the second term on the r.h.s. of \eqref{eq:temp1}.
\begin{eqnarray}
\tr{(\tilde{X} - Z_t^*)^\top G_t} 
& \le  & \NM{\tilde{X} - Z_t^*}{F} \NM{G_t}{F} 
\notag \\
& \le  & \gamma_t \NM{\tilde{X} - Z_t^*}{F}  
\label{eq:temp2}\\
& \le  & \gamma_t \beta_t \NM{Q Q^{\top} - U_k U_k^{\top}}{F}  
\label{eq:temp4} \\
& \le & \eta_t^J (\alpha_t \beta_t \gamma_t).
\label{eq:temp11}
\end{eqnarray}
Here, \eqref{eq:temp2} follows from Proposition~\ref{pr:subbound};
\eqref{eq:temp4} from Lemma~\ref{lem:temp1};
and \eqref{eq:temp11} from Lemma~\ref{lem:powermethod}.
Result follows on combining \eqref{eq:temp1} and \eqref{eq:temp11}.
\end{proof}

\begin{lemma} \label{pr:temp2}
If $\{F(X_t)\}$ is upper-bounded where $F$ is the objective at \eqref{eq:mcgen}, 
then $\NM{X_t}{F}$ from Algorithm~\ref{alg:AISimpute} is upper-bounded.
\end{lemma}

\begin{proof}

As $\{ F(X_t) \}$ is upper bounded
and note that
\begin{align*}
F(X) \rightarrow +\infty
\Leftrightarrow
\NM{X}{F} \rightarrow + \infty.
\end{align*}
for \eqref{eq:mcgen},
then $\{ \NM{X_t}{F} \}$ is also upper bounded.
\end{proof}

%
%

Now, we are ready to prove Proposition~\ref{pr:inexact}.
As $\alpha_t$, $\beta_t$ and $\gamma_t$ only depend on $X_t$, 
from Lemma~\ref{pr:temp2},
they are all upper bounded.
Let $q = \sup_t \alpha_t\beta_t\gamma_t$, and
$q < \infty$ is a constant. 
Then by Proposition~\ref{pr:proxrate}, and note that Algorithm~\ref{alg:apprSVT} is run for $t$ iterations at $t$th loop of Algorithm~\ref{alg:AISimpute}.
Let $\eta = \max_t \eta_t \in (0, 1)$, we have
\begin{align*} 
h_{\mu\lambda \NM{\cdot}{*}}(X_{t + 1} ; \tZ_t) 
\le h_{\mu\lambda \NM{\cdot}{*}}(Z_t^*;\tZ_t) + \varepsilon_t.
\end{align*} 
Hence, $\varepsilon_t = q \eta^t$ decays at a linear rate.

%
%
%
%
%


\subsection{Theorem~\ref{the:AISImpute:conv}}
\label{app:AISImpute:conv}

\begin{proof}
From Proposition~\ref{pr:inexact},
$\varepsilon_t$ decays at a linear rate.
Moreover, 
there is no error on the computation of gradient.
Thus, conditions in Proposition~\ref{cor:iapg:require} are satisfied,
and Algorithm~\ref{alg:AISimpute} 
converges 
with a rate of $O(1/T^2)$.
\end{proof}


\subsection{Proposition~\ref{pr:gsvt}}
\label{app:gsvt}

\begin{proof}
Note that
\begin{align}
& \!\! \min_{\tX^1, \dots, \tX^D} \!
\frac{1}{2}\NM{[\tX^1,\dots,\tX^D] - [\breve{\ten{Z}}_t^1,\dots,\breve{\ten{Z}}_t^D]}{F}^2
\! + \! \mu \sum_{d = 1}^D \lambda_d \NM{\tX^d_{\ip{d}}}{*}
\notag \\
& = \sum_{d = 1}^D \min_{\tX^d}
\frac{1}{2}\NM{\tX^d - \breve{\ten{Z}}_t^d}{F}^2
+ \mu \lambda_d \NM{\tX^d_{\ip{d}}}{*},
\notag \\
& = \sum_{d = 1}^D \min_{\tX^d}
\frac{1}{2}\NM{\tX^d_{\ip{d}} - ( \breve{\ten{Z}}_t^d )_{\ip{d}}}{F}^2
+ \mu \lambda_d \NM{\tX^d_{\ip{d}}}{*}.
\label{eq:temp17}
\end{align}
The $\tX^d$'s in \eqref{eq:temp17} are independent of each other,
and 
\begin{align*}
( \svt_{\mu \lambda_d}(\breve{\ten{Z}}_{\ip{d}}^d) )_{\ip{d}}
\! = \! \arg\min_{\tX^d}
\frac{1}{2}\NM{\tX^d_{\ip{d}} - \breve{\ten{Z}}^d_{\ip{d}}}{F}^2
\! + \! \mu \lambda_d \NM{\tX^d_{\ip{d}}}{*}.
\end{align*}
and thus result follows.
\end{proof}


\subsection{Proposition~\ref{pr:lipten}}
\label{app:lipten}

\begin{proof}
	For any $\ten{X}^1, \dots, \ten{X}^D$, $\ten{Y}^1, \cdots ,\ten{Y}^D$, 
	and let $\tilde{\ten{X}} = \sum_{d = 1}^D \ten{X}^d$ and $\tilde{\ten{Y}} = \sum_{d = 1}^D \ten{Y}^d$.
\begin{eqnarray*}
\lefteqn{\NM{\nabla f([\ten{X}^1, \dots, \ten{X}^D]) - \nabla f([\ten{Y}^1, \cdots
,\ten{Y}^D])}{F}^2}
	\\
&	= & \!\!\!\!\!\!\!\!\!\!\!\!\! \sum_{(i_1,\dots,i_D)\in\Omega} 
	\left[  \frac{d \ell( \tilde{\ten{X}}_{i_1\dots i_D}, \ten{O}_{i_1\dots i_D} )}{d \tilde{\ten{X}}_{i_1\dots i_D}} 
	- \frac{d \ell( \tilde{\ten{Y}}_{i_1\dots i_D}, \ten{O}_{i_1\dots i_D} )}{d \tilde{\ten{Y}}_{i_1\dots i_D}} \right]^2
	\\
	& \le & \!\!\!\!\!\!\!\!\!\!\!\!\! \sum_{(i_1,\dots,i_D)\in\Omega} 
	\rho^2 \left( \tilde{\ten{X}}_{i_1\dots i_D} - \tilde{\ten{Y}}_{i_1\dots i_D} \right)^2
	\le \rho^2 \left\| \tilde{\ten{X}} - \tilde{\ten{Y}} \right\|_F^2,
\end{eqnarray*}  
where the first inequality comes from the $\rho$-Lipschitz smoothness of $\ell$.
Note that
\begin{eqnarray*}
\left\| \tilde{\ten{X}} - \tilde{\ten{Y}} \right\|_F^2
	& \le & D \sum_{d = 1}^D \NM{\ten{X}^d - \ten{Y}^d}{F}^2
	\\
	& = & D \NM{[\ten{X}^1, \dots, \ten{X}^D] - [\ten{Y}^1, \dots, \ten{Y}^D]}{F}^2.
	\end{eqnarray*} 
We have 
\begin{eqnarray*}
\lefteqn{\NM{\nabla f([\ten{X}^1, \dots, \ten{X}^D]) - \nabla f([\ten{Y}^1, \dots,
\ten{Y}^D])}{F}} \\
& \le & \sqrt{D} \rho \NM{[\ten{X}^1, \dots, \ten{X}^D] - [\ten{Y}^1, \dots, \ten{Y}^D]}{F},
\end{eqnarray*}
and thus $f$ is $\sqrt{D}\rho$-Lipschitz smooth.
\end{proof}


\subsection{Theorem~\ref{the:AISImpute:ten:conv}}
\label{app:AISImpute:ten:conv}

\begin{proof}
From the definition of $h$ in \eqref{eq:h},
\begin{eqnarray} 
\lefteqn{h_{\mu g}\left( 
[\ten{X}^1_{t + 1}, \dots, \ten{X}^D_{t + 1}] ; [\breve{\ten{Z}}^1_t, \dots, \breve{\ten{Z}}^D_t] \right)}
\notag \\
& = & \sum_{d = 1}^D 
\frac{1}{2}\left\| (\tX^d_{t + 1})_{\ip{d}} - (\breve{\ten{Z}}_t)^d_{\ip{d}} \right\|_F^2
+ \mu \lambda_d \NM{(\tX^d_{t + 1})_{\ip{d}}}{*},
\notag \\
& = & \sum_{d = 1}^D 
h_{\mu \lambda_d \NM{\cdot}{*}}\left( (\ten{X}^d_{t + 1})_{\ip{d}} ; (\breve{\ten{Z}}_t)^d_{\ip{d}}\right).
\label{eq:temp18}
\end{eqnarray} 
As proximal step is inexact in Algorithm~\ref{alg:AISImpute:ten},
using Proposition~\ref{pr:proxrate} on \eqref{eq:temp18}, 
\begin{eqnarray*}
\lefteqn{h_{\mu \lambda_d \NM{\cdot}{*}}\left( (\ten{X}^d_{t + 1})_{\ip{d}} ;
(\breve{\ten{Z}}_t)^d_{\ip{d}}\right) }
\\
& \le h_{\mu \lambda_d \NM{\cdot}{*}}\left( ( \ten{W}_*^d)_{\ip{d}}; (\breve{\ten{Z}}_t)^d_{\ip{d}} \right) 
+ (\alpha_d)_t (\beta_d)_t (\gamma_d)_t (\eta_d)_t^J,
\end{eqnarray*}
where $(\ten{W}_*^d)_{\ip{d}} = \svt_{\mu \lambda_d \NM{\cdot}{*}}\left( (\breve{\ten{Z}}_t)^d_{\ip{d}} \right)$,
and $\alpha_d$, $\beta_d$, $\gamma_d$, $\eta_d$ are constants depending on $(\breve{\ten{Z}}_t)^d_{\ip{d}}$. 
Let $(c_d)_t = (\alpha_d)_t (\beta_d)_t (\gamma_d)_t$.
As $J = t$, 
\begin{eqnarray} 
\lefteqn{h_{\mu g}\left( [\ten{X}^1_{t + 1}, \cdots, \ten{X}^D_{t + 1}] ; [\breve{\ten{Z}}^1_t, \cdots,
\breve{\ten{Z}}^D_t] \right) }
\label{eq:temp19} \\ 
& \le & h_{\mu g}\left( [\ten{W}_*^1, \cdots, \ten{W}_*^D] ; [\breve{\ten{Z}}^1_t, \cdots, \breve{\ten{Z}}^D_t] \right) 
+ \sum_{d = 1}^D (c_d)_t (\eta_d)_t^t.
\notag
\end{eqnarray}
As $F([\ten{X}^1_t, \dots, \ten{X}^D_t])$ is upper-bounded and 
\begin{align*}
\lim_{\NM{\ten{X}^d}{F} \rightarrow \infty} F([\ten{X}^1_t, \dots, \ten{X}^D_t]) = \infty,
\end{align*}
for any $d = 1, \dots, D$.
Then, $\NM{\ten{X}_t^d}{F}$ for $d = 1, \dots, D$ are also upper-bounded.
Thus,
\begin{align*}
q = \sup_t \sum_{d = 1}^D (c_d)_t < \infty.
\end{align*}
Let $\eta = \max_{d, t}((\eta_d)_t) < 1$.
Together with \eqref{eq:temp19}, we have
\begin{eqnarray*} 
\lefteqn{h_{\mu g}\left( [\ten{X}^1_{t + 1}, \dots, \ten{X}^D_{t + 1}] ; [\breve{\ten{Z}}^1_t, \dots,
\breve{\ten{Z}}^D_t] \right) }
\\
& \le & h_{\mu g}\left( [\ten{W}_*^1, \dots, \ten{W}_*^D] ; [\breve{\ten{Z}}^1_t, \dots, \breve{\ten{Z}}^D_t] \right) 
+ \varepsilon_t,
\end{eqnarray*}
and the approximation error $\varepsilon_t = q \eta^t$ decays at a linear rate.
Moreover, 
there is no error on the computation of gradient.
Thus, the conditions in Proposition~\ref{cor:iapg:require} are satisfied,
and Algorithm~\ref{alg:AISImpute:ten} 
converges 
with a rate of $O(1/T^2)$.
\end{proof}

%

\end{document}